\documentclass[reqno,a4paper,11pt]{amsart}
\usepackage{color}
\usepackage{amsmath,amssymb,amsfonts,amsthm,enumerate}
\usepackage{hyperref,mathrsfs,accents}
\usepackage[T1]{fontenc}
\usepackage[active]{srcltx}
\usepackage{epsfig}
\usepackage{graphicx}
\usepackage{cite}
\usepackage{tikz}
\usepackage{pgfplots}
\usetikzlibrary{shapes,arrows,shapes.geometric,patterns,fadings}
\usepackage{subfig}

\catcode`@=11 \@addtoreset{equation}{section} \catcode`@=12

\newtheorem{theorem}{Theorem}[section]
\newtheorem{lemma}[theorem]{Lemma}
\newtheorem{definition}[theorem]{Definition}
\newtheorem{remark}[theorem]{Remark}
\newtheorem{proposition}[theorem]{Proposition}
\newtheorem{corollary}[theorem]{Corollary}
\newtheorem{example}[theorem]{Example}

\newcommand{\N}{\mathbb{N}}
\newcommand{\R}{\mathbb{R}}

\newcommand{\Sph}{\mathbb{S}} %Sphere
\newcommand{\Mm}{\mathcal{M}}

\newcommand{\MH}{\mathcal{M_{\Haus{}}}}

\newcommand{\Jj}{\mathfrak{J}}

\newcommand{\eps}{\varepsilon}

\newcommand{\weakto}{\rightharpoonup}

\newcommand{\Haus}[1]{{\mathcal H}^{#1}} % Misura di Hausdorff
\newcommand{\Leb}[1]{{\mathcal L}^{#1}} % Misura di Lebesgue
\newcommand{\Per}{P} %Perimeter

\newcommand{\mean}[1]{\,-\hskip-1.08em\int_{#1}} %averaged integral
\newcommand{\DM}{\mathcal{DM}}
\newcommand{\redb}{\partial^{*}} %reduced boundary 

\newcommand{\ulk}{{u^\lambda_k}}
\newcommand{\ul}{{u^\lambda}}

\newcommand{\mul}{M[u,\lambda]}
\newcommand{\Tr}{\mathrm{Tr}}

\newcommand{\de}{\partial}

\newcommand{\restrict}{\mathbin{\vrule height 1.4ex depth 0pt width
0.13ex\vrule height 0.13ex depth 0pt width 1.3ex}\,}

\newcommand{\B}{\mathcal{B}}
\newcommand{\cF}{{\mathcal F}}

\newcommand{\T}{\mathfrak{T}}
\newcommand{\PB}{\mathcal{PB}}

\renewcommand{\Subset}{\subset\!\subset}

\AtBeginDocument{
  \let\div\relax
  \DeclareMathOperator{\div}{div}
}
\newcommand{\res}{\mathop{\hbox{\vrule height 7pt width .5pt depth 0pt
\vrule height .5pt width 6pt depth 0pt}}\nolimits}

\definecolor{grey}{rgb}{.7,.7,.7}
\definecolor{evidGP}{rgb}{0,0,1}
\definecolor{evidG}{rgb}{0,0.5,0}

\title[Measures in the dual of $BV$]{Measures in the dual of $BV$: perimeter bounds and relations with divergence-measure fields}

\author[Giovanni E. Comi]{Giovanni E. Comi}
\address[Giovanni E. Comi]{Dipartimento di Matematica, Università di Bologna, Piazza di Porta San Donato 5, 40126 Bologna (Italy)}
\email{giovannieugenio.comi@unibo.it}

\author[Gian Paolo Leonardi]{Gian Paolo Leonardi}
\address[Gian Paolo Leonardi]{Dipartimento di Matematica, Università di Trento, via Sommarive 14, IT-38123 Povo - Trento (Italy)}
\email{gianpaolo.leonardi@unitn.it}

\thanks{This work has been financially supported by GNAMPA - INdAM. %In particular, the first-named author was partially supported by the INdAM--GNAMPA 2024 Project \textit{Pairing e div-curl lemma: estensioni a campi debolmente derivabili e differenziazione non locale}, CUP\_E53\-C23\-001\-670\-001. 
Both the authors have been also partially supported by the project PRIN 2017TEXA3H (\textit{Gradient flows, Optimal Transport, and Metric Measure Structures}) and the second-named author by the project PRIN 2022PJ9EFL (\textit{Geometric Measure Theory: Structure of Singular Measures, Regularity Theory and Applications in the Calculus of Variations}). Part of this work was undertaken while the first author visited the University of Trento. He would like to thank this institution for its support and warm hospitality during the visit. The authors are particularly grateful to Lorenzo Brasco and Giorgio Saracco for their support and encouragement during the preparation of the paper. The authors are also grateful to Virginia De Cicco for her comments on a preliminary version of the paper.}

\makeatletter
\@namedef{subjclassname@2020}{\textup{2020} Mathematics Subject Classification}
\makeatother


\begin{document}

\begin{abstract}
We analyze some properties of the measures in the dual of the space $BV$, by considering (signed) Radon measures satisfying a perimeter bound condition, which means that the absolute value of the measure of a set is controlled by the perimeter of the set itself, and whose total variations also belong to the dual of $BV$. We exploit and refine the results of \cite{Phuc_Torres}, in particular exploring the relation with divergence-measure fields and proving the stability of the perimeter bound from sets to $BV$ functions under a suitable approximation of the given measure. As an important tool, we obtain a refinement of Anzellotti-Giaquinta approximation for $BV$ functions, which is of separate interest in itself and, in the context of Anzellotti's pairing theory for divergence-measure fields, implies a new way of approximating $\lambda$-pairings, as well as new bounds for their total variation. These results are also relevant due to their application in the study of weak solutions to the non-parametric prescribed mean curvature equation with measure data, which is explored in a subsequent work.
\end{abstract}

\maketitle

\tableofcontents

\section{Introduction}

Given the ubiquitous presence of the space of functions of bounded variation $BV$ in the Calculus of Variation, it is natural to try and characterize its dual space. An integral representation of the elements of $BV(\Omega)^*$, for any open set $\Omega \subset \R^n$, was obtained in \cite{FuscoSpector2018}, under the Continuum Hypothesis (we refer to the introduction of \cite{FuscoSpector2018} for a detailed account of the previous research related to $BV^*$). Of particular interest is the subspace of $BV(\Omega)^*$ of those linear continuous functionals $\T_{\mu}$ whose action on $u \in BV(\Omega) \cap L^\infty(\Omega)$ can be represented as
\begin{equation} \label{eq:meas_dual_functional}
\T_{\mu} (u) = \int_\Omega u^* \, d \mu,
\end{equation}
where $u^*$ is the precise representative of $u$ and $\mu$ is a suitable finite Radon measure: these functionals are often called Radon measures in the dual of $BV$ \cite{Meyers_Ziemer, phuc2008characterizations, Phuc_Torres}. Given that $u^*$ is well defined up to a set with zero $(n-1)$-dimensional Hausdorff measure $\Haus{n-1}$, an immediate property of the measures $\mu$ in the dual of $BV$ is that $|\mu|(B) = 0$ for every Borel set with $\Haus{n-1}(B) = 0$.  As explored in \cite{phuc2008characterizations, Phuc_Torres}, under mild assumptions on the regularity of $\Omega$ we have that a finite Radon measure $\mu$ belongs to $BV(\Omega)^*$ in this sense if and only if $\mu = \div F$ for some vector field $F \in L^\infty(\Omega; \R^n)$. This shows a natural connection between the Radon measures in the dual of the space $BV$ and the divergence-measure fields.

Given $p \in [1, +\infty]$, a $p$-summable divergence-measure field on an open set $\Omega$ is a vector field $F \in L^{p}(\Omega; \R^n)$ such that its distributional divergence $\div F$ is a finite Radon measure on $\Omega$. We denote the space of such fields as $\DM^{p}(\Omega)$. Initially introduced by Anzellotti in \cite{Anzellotti_1983}, the divergence-measure fields have been extensively studied, among other motivations, because of their natural applications in relation with generalized Gauss--Green formulas and weak formulations of some families of PDEs (see for instance \cite{ACM,ChenComiTorres,CF1,CDS,crasta2017extension,comi2017locally,crasta2017anzellotti,crasta2019pairings,comi2022representation,crasta2022variational,CTZ,phuc2008characterizations,Phuc_Torres,scheven2016bv,scheven2018dual,Silhavy1,Silhavy2, MR4385590}). Under this respect, it is important to note that, when $p \in [1, + \infty)$, these fields can be naturally paired with scalar functions $u \in W^{1, p'}(\Omega)$, where $p'$ is the conjugate exponent of $p$, since the scalar product $F \cdot \nabla u$ belongs to $L^1(\Omega)$. On the other hand, the case $p = +\infty$ is more complex and interesting to study, since $p' =1$, which includes the case of $u \in BV(\Omega)$.
Following the approach of \cite{crasta2019pairings}, for a given Borel function $\lambda:\Omega\to [0,1]$, we can define the $\lambda$-pairing distribution between a vector field $F\in \DM^{\infty}(\Omega)$ and a scalar function $u\in BV(\Omega)$ as follows:
\[
(F, Du)_{\lambda} := \div(uF) - u^{\lambda}\, \div F\,,
\]
assuming $u^\lambda \in L^1(\Omega; |\div F|)$, where $u^{\lambda}$ is defined $\Haus{n-1}$-almost everywhere as the convex combination of $u^{+}$ and $u^{-}$ (the upper and lower approximate limits of $u$) using the coefficients $\lambda$ and $1-\lambda$, respectively. We point out that this expression is well posed, since sets with zero $\Haus{n-1}$-measure are also $|\div F|$-negligible. As proved in \cite{crasta2019pairings}, the $\lambda$-pairing is a distribution of order zero, i.e., it is a (finite) Radon measure. Moreover, the classical pairing introduced by Anzellotti in \cite{Anzellotti_1983} corresponds to $\lambda \equiv \frac{1}{2}$, and in this case $u^{\lambda}$ coincides $\Haus{n-1}$-almost everywhere with $u^{*}$, the \textit{precise representative} of $u$. For this reason, we set $(F, Du)_* := (F, Du)_{\frac{1}{2}}$. It is also clear that $(F, Du)_* = (F, Du)_\lambda$ for all $u \in BV(\Omega)$ such that $u^* \in L^1(\Omega; |\div F|)$ as long as the Borel function $\lambda: \Omega \to [0, 1]$ satisfies $\lambda(x) = \frac{1}{2}$ for $|\div F|$-a.e. $x \in \Omega$.

Interestingly, every $\lambda$-pairing enjoys the same absolute continuity property with respect to the weak gradient of $u$: 
\begin{equation} \label{eq:pairing_bound_lambda_intro}
|(F, Du)_\lambda| \le \|F\|_{L^\infty(\Omega; \R^n)} |Du| \ \text{ on } \Omega.
\end{equation}
This implies that the mapping
$$BV(\Omega) \cap L^\infty(\Omega) \ni u  \mapsto (F, Du)_\lambda(\Omega)$$ 
defines a functional which is continuous with respect to the $BV$-(semi)norm. As noted in \cite[Remark 4.6]{crasta2019pairings}, the $\lambda$-pairing is linear with respect to $u$ if and only if it coincides with the classical one $(F, Du)_*$; that is, if $\lambda \equiv \frac{1}{2}$ up to $|\div F|$-negligible sets.
Therefore, the mapping
\begin{equation} \label{eq:pairing_functional_intro}
BV(\Omega) \cap L^\infty(\Omega)\ni u \mapsto (F, Du)_*(\Omega)
\end{equation}
is a linear functional, continuous with respect to the $BV$-norm. These observations suggest a relation between the pairings and the functionals in the dual of $BV$.

In addition, in \cite{Phuc_Torres} it was proved that, in the case $\Omega$ is an open bounded set with Lipschitz boundary, a finite Radon measure $\mu$ belongs to $BV(\Omega)^*$ -- that is, the functional $\T_{\mu} : BV(\Omega)\cap L^{\infty}(\Omega) \to \R$ is continuous with respect to the $BV$-norm -- if and only if there exists $C > 0$ such that 
\begin{equation} \label{eq:PB_Intro_PT}
|\mu(U \cap \Omega)| \le C\, \Per(U) \quad \text{ for all open bounded sets } U \subset \R^n \text{ with smooth boundary,}
\end{equation}
where $P(U)$ is the perimeter of the set $U$, which coincides with the surface measure of its boundary, $\Haus{n-1}(\partial U)$, for smooth sets.
Therefore, $\mu \in BV(\Omega)^*$ if and only if the measure $\mu$ enjoys this bound for its absolute value on open smooth sets in terms of their perimeter.

Finally, given $\mu \in BV(\Omega)^*$ there is the non-trivial question of whether $\T_{\mu}(u)$ can be represented as the integral in \eqref{eq:meas_dual_functional} even for $u \in BV(\Omega) \setminus L^\infty(\Omega)$. Indeed, it can happen that $u^* \notin L^1(\Omega; |\mu|)$ (as we show in Remark \ref{rem:counterexample} below), so that the integral representation would not be well-posed.

In this paper, we explore all these known connections recalled so far, in particular integrating and refining some of the results of \cite{Phuc_Torres}.

More precisely, we consider the natural generalization of the condition \eqref{eq:PB_Intro_PT}: we say that a finite Radon measure $\mu \in \Mm(\Omega)$ satisfies a perimeter bound condition if there exists $L > 0$ such that
\begin{equation} \label{eq:PB_intro}
|\mu(E^{1} \cap \Omega)| \le L\, \Per(E)\, \quad \text{ for all measurable sets } \, E\subset \Omega,
\end{equation}
where $E^1$ is the set of points where the Lebesgue density of $E$ is 1, also called the measure theoretic interior of $E$. In this case, we write $\mu \in \PB_{L}(\Omega)$, and we also set $\PB(\Omega) := \bigcup_{L > 0} \PB_{L}(\Omega)$. We notice that in this definition we are testing the perimeter bound condition on the entire family of sets with finite perimeter. We actually prove that, if $\Omega$ is weakly regular (that is, it is a bounded open set such that $\Haus{n - 1}(\partial \Omega) = P(\Omega) < + \infty$) then \eqref{eq:PB_Intro_PT} and \eqref{eq:PB_intro} are equivalent (see Lemma \ref{lem:open_sets_PB} and Proposition \ref{prop:adm_PB_div}(2)). However, our apparently stronger definition is motivated by the fact that the divergence measure of any $F \in \DM^\infty(\Omega)$ satisfies \eqref{eq:PB_intro} for $L = \|F\|_{L^\infty(\Omega; \R^n)}$, due to the generalized Gauss--Green formula: if $E \subseteq \Omega$ is a set of finite perimeter in $\R^n$ such that either $E \Subset \Omega$ or $\Omega$ is weakly regular, then
\begin{equation} \label{eq:GG_intro}
\div F(E^1 \cap \Omega)  = - \int_{\redb E} {\rm Tr}^i(F, \redb E) \, d \Haus{n-1},
\end{equation}
where ${\rm Tr}^i(F, \redb E)$ is the interior normal trace of $F$ and satisfies
\begin{equation*}
\|{\rm Tr}^i(F, \redb E)\|_{L^\infty(\redb E; \Haus{n-1})}  \le \|F\|_{L^\infty(E; \R^n)}.
\end{equation*}
This and similar formulas have been widely studied in the literature \cite{ChenComiTorres, CTZ, comi2022representation, comi2017locally, crasta2017anzellotti, crasta2019pairings, Silhavy1, Silhavy2}, but we present in Section \ref{sec:div_meas} some refinements which are suitable for our purposes.

Overall, we obtain the following refinement of \cite[Theorem 8.2]{Phuc_Torres}: if $\Omega$ is a bounded open set with Lipschitz boundary, then
\begin{equation*}
\mu \in \PB(\Omega) \iff \mu \in BV(\Omega)^* \iff \mu = \div F \ \text{ for some } F \in \DM^{\infty}(\Omega).
\end{equation*}
Thanks to the Gauss-Green formula \eqref{eq:GG_intro}, it is indeed intuitively clear that, if $\mu = \div F$ for some $F \in \DM^{\infty}(\Omega)$, then $\mu \in \PB_L(\Omega)$ for $L = \|F\|_{L^\infty(\Omega; \R^n)}$. However, it is interesting to ask whether some kind of opposite implication holds true; that is, if, given $\mu \in \PB_L(\Omega)$, we can find $F \in \DM^{\infty}(\Omega)$ such that $\mu = \div F$ and $\|F\|_{L^\infty(\Omega; \R^n)} \le L$. Such result would be the natural extension of \cite[Lemma 7.3]{Phuc_Torres}, which states that, if $\Omega$ is a bounded open set with Lipschitz boundary and $\T \in W^{1,1}_0(\Omega)^*$, then there exists $F \in \DM^{\infty}(\Omega)$ such that
\begin{equation} \label{eq:intro_W_11_0}
\T(u) = \int_\Omega u^* \, d \div F \ \text{ for all } u \in W^{1,1}_0(\Omega) \cap L^\infty(\Omega)
\end{equation}
and we have
\begin{equation*}
\|\T\|_{W^{1,1}_0(\Omega)^*} = \min \left \{ \|G\|_{L^{\infty}(\Omega; \R^n)} : G \in \DM^\infty(\Omega) \text{ satisfying } \eqref{eq:intro_W_11_0} \right \}.
\end{equation*}
Therefore, we need to show that, if $\mu \in \PB_L(\Omega)$, then the related functional $\T_\mu$ is well defined on the whole $W^{1,1}_0(\Omega)$ and it satisfies a bound for Sobolev functions similar to the perimeter bound; that is, 
\begin{equation*}
\left | \int_{\Omega} u^{*} \, d \mu \right | \le L \int_{\Omega} |\nabla u| \, dx \ \text{ for all } u \in W^{1,1}_0(\Omega),
\end{equation*}
so that we get
\begin{equation*}
\|\T_\mu\|_{W^{1,1}_0(\Omega)^*} = \sup \left \{ \int_\Omega u^{*} \, d\mu : u \in W^{1,1}_0(\Omega), \| \nabla u \|_{L^1(\Omega; \R^n)} \le 1 \right \} \le L.
\end{equation*}
However, as already noted, it is in general not true that, if $\mu \in \PB(\Omega)$, then $u^* \in L^1(\Omega; |\mu|)$ for unbounded $BV$ or even Sobolev functions (we give an example of such case in Remark \ref{rem:counterexample}).
Therefore, it turns out that we need a stronger assumption on the measure $\mu$: we say that $\mu$ is \textit{admissible}, if $|\mu| \in BV(\Omega)^*$, see Definition \ref{def:muadmissible}. While it might be challenging to verify the admissibility assumption in general, we notice that, if $\Omega$ is an open bounded set with Lipschitz boundary, a relevant subclass of admissible measures consists of those that can be written as $\mu = h \Leb{n} + \gamma \Haus{n-1} \res \Gamma$, where $h \in L^{q}(\Omega)$ for some $q>n$, $\gamma \in L^{\infty}(\Gamma; \Haus{n-1})$ and $\Gamma \Subset \Omega$ is a compact set with finite $\Haus{n-1}$ measure and suitable decay properties on balls (see Example \ref{ex:tMomega}). 
If $\mu \in \Mm(\Omega)$ is admissible, it is indeed true that $u^\lambda \in L^1(\Omega; |\mu|)$ for all $u \in BV(\Omega)$ and $\lambda : \Omega \to [0, 1]$ Borel, thus including also the case $\lambda \equiv \frac{1}{2}$. A key step in the proof is Theorem \ref{thm:smooth_lambda_approx}, which is a refined version of the Anzellotti-Giaquinta approximation for functions $u \in BV(\Omega)$, since it additionally guarantees pointwise $\Haus{n-1}$-a.e. convergence to any given $\lambda$-representative $u^{\lambda}$. It is also worth mentioning that such a result is of interest in itself, particularly concerning the theory of $\lambda$-pairings mentioned earlier, since it provides a new way of deriving \eqref{eq:pairing_bound_lambda_intro} and it allows obtaining a new bound for the pairings in terms of the area functional, see Theorem \ref{lemma:pairlambda-vs-area}; moreover, it plays a fundamental role in several other proofs along the paper. All in all, we can show that, if $\Omega$ is an open bounded set with Lipschitz boundary, $L > 0$ and $\mu \in \PB_L(\Omega)$ is admissible, then $\mu = \div F$ for some $F \in \DM^{\infty}(\Omega)$ with $\|F\|_{L^\infty(\Omega; \R^n)} \le L$ and 
\begin{equation} \label{eq:intro_u_lambda_est}
\left | \int_{\Omega} u^{\lambda} \, d \mu \right | \le L \left(|Du|(\Omega) + \int_{\de \Omega} |{\rm Tr}_{\partial \Omega}(u)|\, d\Haus{n-1}\right) \ \text{ for all } u \in BV(\Omega) \text{ and } \lambda : \Omega \to [0, 1] \text{ Borel,}
\end{equation}
see Proposition \ref{prop:trunc-smooth-coarea} and Lemma \ref{lem:existence_optimal_T}.
This means that the constant $L$ remains unchanged from the perimeter bound for sets to the estimate involving functions of bounded variations, and in this way it refines the results of \cite{Phuc_Torres}.
It is also interesting to notice that, if $\mu \in \Mm(\Omega)$ is admissible, then, for any vector field $F$ such that $\div F = \mu$, the $\lambda$-pairing $(F, Du)_{\lambda}$ is well-defined for all $BV$ function $u$ and all Borel functions $\lambda$ with values in $[0,1]$ (see Remark \ref{rem:div_admissible_Leibniz}). In addition, if $\Omega$ is an open bounded set with Lipschitz boundary, the action of the functional represented by an admissible measure $\mu$ can be expressed in terms of the pairing between $F \in \DM^\infty(\Omega)$, such that $\mu = \div F$, and functions in $BV(\Omega)$; that is,
\begin{equation*}
\T_{\mu} (u) = \int_\Omega u^* \, d \mu = - (F, Du)_*(\Omega) - \int_{\partial \Omega} {\rm Tr}_{\partial \Omega}(u) {\rm Tr}^i(F, \partial \Omega) \, d \Haus{n-1} \ \text{ for all } u \in BV(\Omega)
\end{equation*}
where ${\rm Tr}_{\partial \Omega}(u) \in L^1(\partial \Omega; \Haus{n-1})$ and ${\rm Tr}^i(F, \partial \Omega) \in L^\infty(\partial \Omega; \Haus{n-1})$ are the interior trace of $u$ and the interior normal trace of $F$ on the boundary of $\Omega$, respectively. This fact relates to the previous observations on the pairing functional \eqref{eq:pairing_functional_intro}, since, if $\mu = \div F$ for some $F \in \DM^\infty(\Omega)$, then 
\begin{equation*}
\T_{\mu} (u) = - (F, Du)_*(\Omega) \ \text{ for all } u \in BV_0(\Omega);
\end{equation*}
that is, $u \in BV(\Omega)$ with ${\rm Tr}_{\partial \Omega}(u) = 0$.

Finally, we prove that the perimeter bound and the admissibility conditions combined ensure enough stability under a suitable type of smooth approximation procedure (Proposition \ref{prop:muapprox}). More precisely, if $\Omega$ is an open bounded set with Lipschitz boundary and $\mu \in \PB_{L}(\Omega)$ is an admissible measure, then $\mu = \div F$ for some $F \in \DM^\infty(\Omega)$ with $\|F\|_{L^\infty(\Omega; \R^n)} \le L$ and so, relying on the fact that the Anzellotti-Giaquinta--type regularization of a vector field almost preserves its $L^{\infty}$ norm, we can approximate $\mu$ in the weak--$\ast$ sense by a sequence of absolutely continuous measures $\mu_j \Leb{n} \in \PB_{L_j}(\Omega)$ with smooth density functions, where $L_j \to L$ as $j \to + \infty$.

All these facts are of great relevance in their application to the study of the prescribed mean curvature measure equation  
\begin{equation} \label{eq:PMCM}  \div \left (\frac{\nabla u}{\sqrt{1 + |\nabla u|^{2}}} \right) = \mu \quad \text{on} \ \Omega, \end{equation}
for an admissible measure $\mu \in \PB_L(\Omega)$ for some $L \in (0, 1)$, which is the core of our subsequent work \cite{LeoComi}. Following the approach of \cite{scheven2016bv}, we consider the following weak formulation of the equation \eqref{eq:PMCM}.

\begin{definition}
We say that $u \in BV(\Omega)$ is a \textit{weak solution to the prescribed mean curvature measure equation} if there exist $T \in L^{\infty}(\Omega;\R^{n})$ and a Borel function $\lambda:\Omega\to [0,1]$ such that
\begin{align*} 
 \|T\|_{L^{\infty}(\Omega; \R^{n})} & \le 1,\\
\div T & = \mu\, \text{ on } \Omega, \\
 (T, Du)_{\lambda} & = \sqrt{1 + |D u|^{2}} - \sqrt{1 - |T|^{2}} \Leb{n} \, \text{ on } \Omega,
\end{align*}
where the last two identities involve scalar Radon measures in $\Mm(\Omega)$ and $\lambda=\lambda_{\mu}$ is the characteristic function of a Borel set satisfying $(1-\lambda)\mu = \mu^{+}$, where $\mu^\pm$ are the positive and negative parts of $\mu$.
\end{definition}

The existence of weak solutions of \eqref{eq:PMCM} is obtained via the direct method of Calculus of Variations, which involves proving coercivity and lower semicontinuity for the functional 
\[
\Jj_{\mu}[u] := \sqrt{1+|Du|^{2}}(B) + \int_{\Omega}u^{-}\, d\mu^{+} - \int_{\Omega} u^{+}\, d\mu^{-}\,,
\]
In particular, \eqref{eq:intro_u_lambda_est} directly implies that $\Jj_{\mu}$ is coercive, and the stability of the admissibility and perimeter bound conditions are exploited in order to prove a Gamma-convergence result for a sequence of functionals $\Jj_{\mu_{j}}$ (see \cite[Theorem 6.2]{LeoComi}).

\section{Preliminaries}
\label{sec:prelim}

Through the rest of the paper, we work in an open set $\Omega \subset \R^n$. We denote by $\Leb{n}$ the Lebesgue measure, and by $\Haus{m}$ the $m$-dimensional Hausdorff measure, for $m \in [0, n]$, although we shall focus on the case $m = n-1$. Given $x \in \R^n$ and $r > 0$, we denote by $B_r(x)$ the open ball centered in $x$ with radius $r$, and, if $x = 0$, we simply write $B_r$.
We denote by $\Mm(\Omega)$ the space of finite Radon measures on $\Omega$, and by $\MH(\Omega)$ the space of measures in $\Mm(\Omega)$ that are absolutely continuous with respect to $\Haus{n-1}$; that is, 
\begin{equation}\label{def:MH}
\MH(\Omega) := \{ \mu \in \Mm(\Omega) : |\mu|(B) = 0 \text{ for all Borel sets } B \subset \Omega \text{ such that } \Haus{n-1}(B) = 0 \}.
\end{equation}
Given two measures $\mu_1, \mu_2 \in \Mm(\Omega)$, we say that $\mu_1 \le \mu_2$ on $\Omega$ if $\mu_1(B) \le \mu_2(B)$ for all Borel sets $B \subseteq \Omega$. If $\mu \in \Mm(\Omega)$ satisfies $\mu \ge 0$ on $\Omega$, then we say that $\mu$ is nonnegative. Thanks to the Hahn's decomposition theorem, we know that for any $\mu \in \Mm(\Omega)$ there exist two nonnegative measures $\mu^+$ and $\mu^-$ (the positive and negative part of $\mu$, respectively), which satisfy $\mu = \mu^+ - \mu^-$ and are concentrated on mutually disjoint Borel sets $\Omega_+,\Omega_-$ such that $\Omega = \Omega_+ \cup \Omega_{-}$. In particular, we have $\mu_{+} = \mu \restrict \Omega_+$, $\mu_{-} = -\mu \restrict \Omega_-$, and $|\mu| = \mu^+ + \mu^-$. In addition, by Lebesgue-Besicovitch differentiation theorem, for $|\mu|$-a.e. $x \in \Omega$ we have
\begin{equation} \label{eq:mu_pm_differentiation}
\frac{d \mu^\pm}{d |\mu|}(x) = \begin{cases} 1 & \text{ if } x \in \Omega^{\pm} \\
0 & \text{ if } x \in \Omega^{\mp}.
\end{cases}
\end{equation}

Given $\mu \in \Mm(\Omega)$, Radon-Nikodym theorem and Lebesgue's decomposition theorem ensure that we can decompose it as $$\mu = \mu^{ac} + \mu^s$$
where $\mu^{ac}= g \Leb{n}$, for some $g \in L^1(\Omega)$, is the absolutely continuous part and $\mu^s$ is the singular part. 

\subsection{Functions of bounded variation} We say that $u \in L^1(\Omega)$ is a function of bounded variation, and we write $u \in BV(\Omega)$, if its distributional gradient $Du$ is a vector valued Radon measure whose total variation $|Du|$ is a finite measure on $\Omega$. $BV(\Omega)$ is a Banach space once equipped with the norm
$$ \|u\|_{BV(\Omega)} = \|u\|_{L^1(\Omega)} + |Du|(\Omega).$$

However, the convergence induced by this norm is too strong, and hence it is customary to consider a weaker type of convergence for sequences of $BV$ functions, the {\em strict convergence}: given $u \in BV(\Omega)$, we say that a sequence $(u_j)_{j \in \N}$ converges to $u$ in $BV(\Omega)$-strict if
\begin{equation*}
\lim_{j \to + \infty} \|u - u_j\|_{L^1(\Omega)} + \left | |Du|(\Omega) - |Du_j|(\Omega) \right | = 0.
\end{equation*}
As customary, we denote by $BV_{\rm loc}(\Omega)$ the local version of the space $BV(\Omega)$.

Following the notation of \cite[Section 3.6]{AFP}, we say that a function $u \in L^{1}_{\rm loc}(\Omega)$ has approximate limit at $x \in \Omega$ if there exists $z \in \R$ such that
\begin{equation} \label{eq:approx_lim_def}
\lim_{r \to 0} \mean{B_r(x)} |u(y) - z| \, dy = 0\,,
\end{equation}
and we denote by $\widetilde{u}(x)$ the value, $z$, of the approximate limit of $u$ at $x$. In this case, we say that $x$ is a Lebesgue point of $u$. For a given measurable set $E \subset \R^n$, we define the {\em measure theoretic interior} of $E$ as
\[
E^1 := \left \{ x \in \R^n : \lim_{r \to 0} \frac{|E \cap B_r(x)|}{|B_r(x)|} = 1 \right \} = \left \{ x \in \R^n : \widetilde{\chi_E}(x) = 1 \right \}.
\] 
 
As customary, the approximate discontinuity set $S_{u}$ is defined as the set of points where the approximate limit does not exist. In addition, we say that $x$ belongs to $J_u$ (the set of approximate jump points of $u$) if there exists $a, b \in \R$, $a \neq b$, and $\nu \in \Sph^{n - 1}$ such that
\begin{equation} \label{eq:approx_jump_def}
\lim_{r \to 0} \mean{B^{+}_{r}(x, \nu)} |u(y) - a| \, dy = 0 \ \text{ and } \ \lim_{r \to 0} \mean{B^{-}_{r}(x, \nu)} |u(y) - b| \, dy = 0,
\end{equation}
where $B^{\pm}_{r}(x, \nu) := \{ y \in B_{r}(x) : \pm ( y - x) \cdot \nu \ge 0 \}$. The triplet $(a, b, \nu)$ is uniquely determined by \eqref{eq:approx_jump_def} up to a permutation of $(a, b)$ and a change of sign of $\nu$, and we denote it by $(u^{+}(x), u^{-}(x), \nu_{u}(x))$. For the approximate traces $u^{\pm}(x)$, we adopt the convention of having $u^{+}(x) > u^{-}(x)$. We can actually extend the approximate traces also for $x \in \Omega \setminus S_{u}$, by setting $u^{+}(x) = u^{-}(x) = \widetilde{u}(x)$.

Given $u \in BV_{\rm loc}(\Omega)$, following \cite[Corollary 3.80]{AFP} we define its precise representative 
$$u^{*} : \Omega \setminus (S_u \setminus J_u) \to \R$$ by setting
\begin{equation*}
u^*(x) = \begin{cases} \widetilde{u}(x) & \text{ if } x \in \Omega \setminus S_u, \\
\displaystyle \frac{u^{+}(x) + u^{-}(x)}{2} & \text{ if } x \in J_u. \end{cases}
\end{equation*}
Since $\Haus{n - 1}(S_u \setminus J_u) = 0$ by \cite[Theorem 3.78]{AFP}, $u^{*}(x)$ is well defined for $\Haus{n - 1}$-a.e. $x \in \Omega$ and satisfies
\begin{equation} \label{eq:u_star_mean_lim}
u^*(x) = \lim_{r \to 0} \mean{B_{r}(x)} u(y) \, dy \ \text{ for } \Haus{n-1}\text{-a.e. } x \in \Omega.
\end{equation}
In addition, if we consider the extensions of the approximate traces, we also get
\begin{equation} \label{eq:u_star_u_pm_def}
u^{*}(x) = \frac{u^{+}(x) + u^{-}(x)}{2} \ \text{ for } \Haus{n-1}\text{-a.e. } x \in \Omega.
\end{equation}
In this way, we can see $u^*$ as the average between $u^+$ and $u^-$. Therefore, it seems natural to generalize $u^{*}$ by taking any convex combination of the approximate traces. More precisely, for a fixed a Borel function $\lambda:\Omega\to [0,1]$ we define the $\lambda$-representative of $u$ as
\[
u^\lambda = \lambda u^+ + (1-\lambda) u^- \ \text{ on } \Omega \setminus (S_u \setminus J_u),
\]
which is well-defined for $\Haus{n-1}$-a.e. $x \in \Omega$.
Clearly, such representative could be defined for any $\lambda \in \B_{b}(\Omega)$, that is, any bounded Borel function $\lambda : \Omega \to \R$, but only convex combinations of $u^{\pm}$ have relevant properties in relation with generalized notions of pairings between divergence-measure fields and scalar functions of bounded variation. 

We recall the standard decomposition of the distributional gradient of a function $u \in BV(\Omega)$: we have
\begin{equation*}
Du = \nabla u \, \Leb{n} + D^c u + D^j u,
\end{equation*}
where $\nabla u \in L^1(\Omega; \R^n)$ is the density of the absolutely continuous part of $Du$, $D^j u$ is the jump part, that is, $$D^j u = (u^+ - u^-) \nu_u \, \Haus{n-1} \res J_u,$$ and $D^c u$ is the Cantor part. In addition, we define the diffuse and singular parts of $Du$ as
\begin{equation*} 
D^d u = \nabla u \, \Leb{n} + D^c u \ \text{ and } \ D^{s} u = D^c u + D^j u.
\end{equation*}

Let $N > 0$ and $T_{N}$ be the truncation operator, that is, the $1$-Lipschitz map defined as
\begin{equation*}
T_{N}(t) = \begin{cases} N & \text{if} \ \  t > N, \\
t & \text{if} \ \ |t| \le N, \\
- N & \text{if} \ \ t < - N\,,
\end{cases}
\end{equation*}
then extended to functions by setting $T_{N}(u)(x) = T_{N}(u(x))$. 
We recall the statement of \cite[Proposition 3.69 (c)]{AFP}: if $u \in L^{1}_{\rm loc}(\Omega)$ and $x \in J_{u}$, then $x \in J_{T_{N}(u)}$ if and only if $T_{N}(u^{+})(x) \neq T_{N}(u^{-})(x)$, and in this case we have $T_{N}(u)^{\pm}(x) = T_{N}(u^{\pm})(x)$. Otherwise, $x \notin S_{T_{N}(u)}$ and $$\widetilde{T_{N}(u)}(x) = T_{N}(\widetilde{u})(x) = T_{N}(u^{+})(x) = T_{N}(u^{-})(x).$$
Hence, if we extend the approximate traces of $T_{N}(u)$ as above, we get
\begin{equation*}
T_{N}(u)^{\pm}(x) = T_{N}(u^{\pm})(x) \text{ for all } x \in \Omega \setminus (S_{u} \setminus J_{u}).
\end{equation*}
%In particular, this means that, if $x \in J_{u}$, $u^{-}(x)< N$ and $u^{+}(x) > - N$, then $x \in J_{T_{N}(u)}$ and $T_{N}(u)^{\pm}(x) = T_{N}(u^{\pm}(x))$. Instead, if $u^{-}(x) \ge N$, then $T_{N}(u)^{\pm}(x) = N$; and, if $u^{+}(x) \le - N$, then $T_{N}(u)^{\pm}(x) = - N$.
By exploiting the fact that
\begin{equation*}
\left | \frac{T_{N}(\alpha) + T_{N}(\beta)}{2} \right | \le \left | T_N \left ( \frac{\alpha + \beta}{2} \right ) \right | \le \left | \frac{\alpha + \beta}{2}  \right |  \ \text{ for all } \alpha, \beta \in \R,
\end{equation*}
we prove that
\begin{equation*}
|T_{N}(u)^{*}| \le |u^{*}|  \ \text{ on } \Omega \setminus (S_{u} \setminus J_{u}).
\end{equation*}

Given $u\in BV(\Omega)$, we have that $T_{N}(u) \in BV(\Omega)$ for all $N > 0$. In particular, we deduce the following well-known estimate on the precise representative of the truncation of $u$:
\begin{equation} \label{eq:Tustar-ustar}
|T_{N}(u)^{*}(x)| \le |u^{*}(x)|  \ \text{ for } \Haus{n - 1}\text{-a.e. } x \in \Omega.
\end{equation}
In addition, the $\lambda$-representative of $T_{N}(u)$ is well-defined, and we can get convergence properties and bounds for $T_{N}(u)^{\lambda}$ that are uniform in $N$ (see Proposition \ref{prop:mul} below). 
However, we notice that we cannot obtain an analogous estimate $|T_{N}(u)^{\lambda}(x)| \le |u^{\lambda}(x)|$ for $\Haus{n-1}$-a.e. $x \in \Omega$, unless $\lambda \in \{0, \frac 12, 1\}$. As an example, let $N > 0$, $\lambda \equiv \frac{1}{3}$ and 
\begin{equation*}
u(x) = \begin{cases} 2N & \text{ if } x \in B_1, \\
- \frac{5}{4} N & \text{ if } x \notin B_1.
\end{cases}
\end{equation*} 
Then, for all $x \in \partial B_1$, we have $u^{+}(x) = 2 N$ and $u^{-}(x) = - \frac{5}{4}N$, so that 
\begin{equation*}
u^{\lambda}(x) = \frac{1}{3} u^{+}(x) + \frac{2}{3} u^{-}(x) = \frac{2}{3} N \left ( 1 - \frac{5}{4} \right ) = - \frac{N}{6} \text{ for all } x \in \partial B_1.
\end{equation*}
On the other hand, notice that 
\begin{equation*}
T_{N}(u)^{\lambda}(x) = \frac{1}{3} T_{N}(u)^+(x) + \frac{2}{3} T_{N}(u)^-(x) = \frac{N}{3} - \frac{2N}{3} = - \frac{N}{3} \text{ for all } x \in \partial B_1,
\end{equation*}
so that we do not have $|T_{N}(u)^{\lambda}| \le |u^{\lambda}|$ on $\partial B_1$. Nevertheless, we are able to obtain the following result (which is a somewhat modified version of \cite[Proposition 3.4]{crasta2019pairings}).

\begin{proposition}\label{prop:mul}
Let $u \in BV_{\rm loc}(\Omega)$, $\lambda:\Omega\to [0,1]$ be a Borel function and $N > 0$. For all $x \in \Omega \setminus (S_u \setminus J_u)$ we set 
\begin{align}\label{eq:upmlambda}
\nonumber
\mul(x) &:= \min \Big\{\lambda(x) |u^{+}(x)| + (1-\lambda(x)) |u^{-}(x)|\ ,\\ 
&\quad\qquad |u^{*}(x)| +\left |\lambda(x) -\frac{1}{2}\right|\big(|u^{+}(x)| + |u^{-}(x)|\big)\Big\}\,.
\end{align}
Then $\mul(x)$ is well defined for $\Haus{n - 1}$-a.e. $x \in \Omega$, and it satisfies
\begin{equation}\label{eq:TuLambda_estimate}
|T_{N}(u)^{\lambda}(x)| \le \mul(x) \text{ and } |u^{\lambda}(x)| \le \mul(x)  \quad \text{for $\Haus{n - 1}$-a.e. $x \in \Omega$.}
\end{equation}
Moreover, we have
\begin{equation} \label{eq:precise_repr_truncated_lambda2}
T_{N}(u)^{\lambda}(x) = u^{\lambda}(x)
\end{equation}
for all $x \in \Omega$ such that $- N \le u^{-}(x) \le u^{+}(x) \le N$, which in turn implies $T_{N}(u)^{\lambda}(x) \to u^{\lambda}(x)$ for $\Haus{n - 1}$-a.e. $x \in \Omega$, as $N\to +\infty$.
\end{proposition}

\begin{proof}
Assuming $x\in \Omega$ is such that $- N \le u^{-}(x) \le u^{+}(x) \le N$, we have $T_{N}(u)^{\pm}(x) = u^{\pm}(x)$. This immediately implies that
\begin{equation*}
T_{N}(u)^{\lambda}(x) = \lambda(x) T_{N}(u)^{+}(x) + (1 - \lambda(x)) T_{N}(u)^{-}(x) = \lambda(x) u^{+}(x) + (1 - \lambda(x)) u^{-}(x) = u^{\lambda}(x)
\end{equation*}
and this shows \eqref{eq:precise_repr_truncated_lambda2}. Since $u^{+}(x)$ and $u^{-}(x)$ are well-defined and finite for $\Haus{n - 1}$-a.e. $x \in \Omega$, we notice that \eqref{eq:precise_repr_truncated_lambda2} easily implies $T_{N}(u)^{\lambda}(x) \to u^{\lambda}(x)$ as $N \to + \infty$ for $\Haus{n - 1}$-a.e. $x \in \Omega$.

Given that $\Haus{n-1}(S_u \setminus J_u) = 0$, it is clear that $\mul(x)$ is well defined for $\Haus{n - 1}$-a.e. $x \in \Omega$. Now, in order to prove the first inequality in \eqref{eq:TuLambda_estimate}, we must show two separate inequalities. The first inequality directly follows from the definition of $T_{N}$:
\begin{align*}
|T_{N}(u)^{\lambda}(x)| &\le \lambda(x) |T_{N}(u)^{+}(x)| + (1-\lambda(x)) |T_{N}(u)^{-}(x)| \\  
&\le \lambda(x) |u^{+}(x)| + (1-\lambda(x)) |u^{-}(x)|\,.
\end{align*}
The second inequality is obtained by adding and subtracting $T_{N}(u)^{*}$ and by using \eqref{eq:Tustar-ustar}, that is, for $\Haus{n-1}$-a.e. $x\in \Omega$ we have
\begin{align*}
|T_{N}(u)^{\lambda}(x)| &\le |T_{N}(u)^{*}(x)| + |T_{N}(u)^{\lambda}(x) - T_{N}(u)^{*}(x)| \\
&\le |u^{*}(x)| + |(\lambda(x) - 1/2) T_{N}(u)^{+}(x) + (1/2 - \lambda(x)) T_{N}(u)^{-}(x)|\\
&\le |u^{*}(x)| + |\lambda(x) - 1/2|\big(|u^{+}(x)| + |u^{-}(x)|\big)\,.
\end{align*}
By combining the two previous estimates, we get \eqref{eq:TuLambda_estimate} as wanted. Finally, we argue in a similar way with $|u^{\lambda}|$ to conclude.
\end{proof}

\begin{remark}
We observe that \eqref{eq:TuLambda_estimate} gives a finer upper bound for $T_{N}(u)^{\lambda}$ than the one given in \cite[Proposition 3.4]{crasta2019pairings}. In addition, we notice that, if $\lambda \equiv \frac{1}{2}$, then \eqref{eq:TuLambda_estimate} reduces to the standard estimate \eqref{eq:Tustar-ustar}. If we consider instead the cases $\lambda \equiv 1$ and $\lambda \equiv 0$, then we have $M[u,1] = |u^+|$ and $M[u, 0] = |u^-|$.
\end{remark}

Given a function $u \in BV(\Omega)$ we define $\sqrt{1 + |Du|^{2}}$ as the distributional area factor of the graph $\{ (x, t) : u(x) = t \} \subset \Omega \times \R$, as done in  \cite{giusti1984minimal}. While the expression makes sense for $\Leb{n}$-a.e. $x \in \Omega$ for a function $u \in W_{\rm loc}^{1, 1}(\Omega)$, when $u \in BV(\Omega)$ we more generally define
\begin{equation*} \int_{U} \sqrt{1 + |Du|^{2}} := \sup \left \{ \int_{\Omega} \eta + u \div \phi \, dx : (\phi, \eta) \in C^{1}_{c}(U; \R^{n} \times \R), | (\phi, \eta) | \le 1 \right \}, \end{equation*}
for any open set $U \subset \Omega$.
It is then easy to see that we have
\begin{equation} \label{area_factor_decomposition_eq} \sqrt{1 + |Du|^{2}} = \sqrt{1 + |\nabla u|^{2}} \Leb{n} + |D^{s} u|. \end{equation}
Given a measurable set $E$, we say that $E$ is a set of finite perimeter in $\Omega$ if $|D\chi_{E}|(\Omega) < + \infty$, and we set $\Per(E;\Omega) = |D\chi_{E}|(\Omega)$ to be the perimeter of $E$ in $\Omega$. When $\Omega = \R^{n}$ we simply write $\Per(E)$. We refer to \cite{AFP, maggi2012sets} for the definition and properties of $\redb E$, the {\em reduced boundary} of $E$, for which one has $\Per(E;\Omega) = \Haus{n -1}(\redb E \cap \Omega)$. In particular, we recall that $\chi_E^+ = \chi_{E^1 \cup \partial^* E}$ and $\chi_E^- = \chi_{E^1}$, so that we deduce the following explicit formulas for the precise and $\lambda$-representatives of $\chi_E$:
\begin{equation} \label{eq:repr_chi_E}
\chi_E^* = \chi_{E^1} + \frac{1}{2} \chi_{\redb E} \ \text{ and } \ \chi_E^\lambda = \chi_{E^1} + \lambda \chi_{\redb E} \quad \Haus{n-1}\text{-a.e. in } \Omega.
\end{equation}
If $n\ge2$, a remarkable property of sets of finite perimeter $E$ in $\R^n$ is the isoperimetric inequality \cite[Theorem 3.46]{AFP}: there exists a constant $c_n > 0$ such that
\begin{equation*}
\min \{|E|, |\R^n \setminus E| \}^{1 - \frac{1}{n}} \le c_n \Per(E).
\end{equation*}
As a consequence, as long as $|\Omega| < + \infty$, for all sets $E \subset \Omega$ of finite perimeter in $\R^n$ we have
\begin{equation*}
|E \cap \Omega| \le |\Omega|^{\frac{1}{n}} |E|^{1 - \frac{1}{n}}  \le c_n |\Omega|^{\frac{1}{n}} \Per(E),
\end{equation*}
which immediately implies
\begin{equation} \label{eq:isoper_omega}
\|\chi_E\|_{BV(\Omega)} = |E \cap \Omega| + \Per(E, \Omega)  \le (c_n |\Omega|^{\frac{1}{n}} + 1) \Per(E).
\end{equation}

If $\Omega$ has bounded Lipschitz boundary, we denote by ${\rm Tr}_{\partial \Omega}(u)$ the trace of $u$ over $\partial \Omega$ and recall that the trace operator ${\rm Tr}_{\partial \Omega}:BV(\Omega)\to L^{1}(\partial \Omega; \Haus{n-1})$ is linear, surjective and continuous with respect to the topology induced by the strict convergence (see for instance \cite[Theorem 3.88]{AFP}). 
For the ease of the reader, we prove a useful result concerning a Cavalieri-type inequality involving the trace operator for $BV$ functions.

\begin{lemma}\label{lem:slicetrace}
Let $\Omega$ be an open, bounded set with Lipschitz boundary. Let $f\in BV(\Omega)$ be a nonnegative function, and set $E_{t}=\{x:\ f(x)>t\}$ for $t>0$. Then
\[
\int_{0}^{+\infty} \int_{\de\Omega} {\rm Tr}_{\de\Omega}(\chi_{E_{t}})\, d\Haus{n-1}\, dt \le \int_{\de\Omega} {\rm Tr}_{\de\Omega}(f)\, d\Haus{n-1}\,.
\]
\end{lemma}
\begin{proof}
By \cite[Theorem 3.87]{AFP}, we have
\begin{align*}
\int_{\de\Omega}{\rm Tr}_{\de\Omega}(f)(x)\, d\Haus{n-1}(x) &= \int_{\de\Omega} \lim_{r\to 0} \frac{1}{|\Omega\cap B_{r}(x)|}\int_{\Omega\cap B_{r}(x)} f(y)\, dy\, d\Haus{n-1}(x)\\
&= \int_{\de\Omega} \lim_{r\to 0} \frac{1}{|\Omega\cap B_{r}(x)|}\int_{\Omega\cap B_{r}(x)} \int_{0}^{+\infty} \chi_{E_{t}}(y)\, dt\, dy\, d\Haus{n-1}(x)\\
&= \int_{\de\Omega} \lim_{r\to 0} \int_{0}^{+\infty} \frac{1}{|\Omega\cap B_{r}(x)|}\int_{\Omega\cap B_{r}(x)} \chi_{E_{t}}(y)\, dy\, dt\, d\Haus{n-1}(x)\\
&\ge \int_{\de\Omega} \int_{0}^{+\infty} \liminf_{r\to 0} \frac{1}{|\Omega\cap B_{r}(x)|}\int_{\Omega\cap B_{r}(x)} \chi_{E_{t}}(y)\, dy\, dt\, d\Haus{n-1}(x)\\
&= \int_{\de\Omega} \int_{0}^{+\infty} {\rm Tr}_{\de\Omega}(\chi_{E_{t}})(x)\, dt\, d\Haus{n-1}(x)\\
&= \int_{0}^{+\infty} \int_{\de\Omega} {\rm Tr}_{\de\Omega}(\chi_{E_{t}})(x)\, d\Haus{n-1}(x)\, dt\,,
\end{align*}
where the inequality in the fourth line follows from Fatou's lemma, while the exchange of the integration orders is a consequence of Tonelli's theorem.
\end{proof}

\subsection{Divergence-measure fields and $\lambda$-pairings} \label{sec:div_meas}
We recall the notions of (essentially bounded) divergence-measure field and of pairing between a function of bounded variation and one such field.

\begin{definition} \label{DMdef}
A vector field $F \in L^{\infty}(\Omega; \R^{n})$ is called an {\em essentially bounded divergence-measure field}, and we write $F \in \DM^{\infty}(\Omega)$,
if $\div F\in \Mm(\Omega)$.
A vector field $F \in L^{\infty}_{\rm loc}(\Omega; \R^n)$ is called a {\em locally essentially bounded divergence-measure field}, and we write $F \in \DM^{\infty}_{\rm loc}(\Omega)$, if $F \in \DM^{\infty}(U)$ for any open set $U \subset \subset \Omega$.
%A vector field $F \in L^{p}(\Omega; \R^{n})$ for some $1 \le p \le \infty$ is called a {\em divergence-measure field}, denoted as $F \in \DM^{p}(\Omega)$,
%if $\div F\in \Mm(\Omega)$.
%A vector field $F$ is a {\em locally divergence-measure field}, denoted as ${F \in \DM^{p}_{\rm loc}(\Omega)}$,
%if the restriction of $F$ to $U$ is in $\DM^{p}(U)$ for any $U \Subset \Omega$ open.
%In the case $p = \infty$, $F$ will be called a {\em (locally) essentially bounded divergence-measure field.}
\end{definition}

It has been proved by Chen and Frid \cite{CF1} (see also \cite{Frid2}, for an improved proof), that, if $F \in \DM^{\infty}(\Omega)$ and $u \in BV(\Omega) \cap L^{\infty}(\Omega)$, then the product $uF$ belongs to $\DM^{\infty}(\Omega)$; and an analogous result holds also locally. Furthermore, we recall the fact that, if $F \in \DM^{\infty}_{\rm loc}(\Omega)$, then $|\div F| \ll \Haus{n - 1}$, for which we refer to \cite{CF1} and \cite{Silhavy1}. Hence, if $u^{*}$ is the precise representative of $u$, the measure $u^{*} \div F$ is well defined, $u^{*}$ being defined $\Haus{n - 1}$-a.e. on $\Omega$.

Actually, the $L^\infty$-assumption on $u$ is not strictly necessary, as one can see by following the approach of \cite{crasta2019pairings} in defining the notion of $\lambda$-pairing.

\begin{definition} \label{BVpairing_lambda_def}
Given a vector field $F \in \DM^{\infty}_{\rm loc}(\Omega)$, a scalar function $u \in BV_{\rm loc}(\Omega)$ and a Borel function $\lambda : \Omega \to [0, 1]$ such that $u^{\lambda} \in L^1_{\rm loc}(\Omega; |\div F|)$, we define the $\lambda$-{\em pairing} between $F$ and $Du$ as the distribution $(F, Du)_\lambda$ given by
\begin{equation} \label{BV_pairing_lambda_def_eq} (F, Du)_\lambda := \div(u F) - u^{\lambda} \div F. \end{equation}
\end{definition}

Roughly speaking, the notion of $\lambda$-pairing extends the classical dot product between $F$ and $Du$. We also remark that in the case $\lambda \equiv \frac 12$ one recovers the classical pairing first considered by Anzellotti in \cite{Anzellotti_1983}, which we denote by $(F, Du)_*$. As proved in \cite{crasta2019pairings}, the $\lambda$-pairing is indeed a Radon measure. For the ease of the reader, we recall its main properties in the following statement, that summarizes \cite[Proposition 4.4]{crasta2019pairings} (see also \cite[Remark 3.6]{comi2022representation}).

\begin{theorem}\label{thm:pairlambda_Leibniz}
Let $F\in \DM^{\infty}(\Omega)$, $u\in BV(\Omega)$ and $\lambda:\Omega\to [0,1]$ be a Borel function. If $u^\lambda \in L^{1}(\Omega; |\div F|)$, then we have $\div(u F), (F, Du)_{\lambda} \in \Mm(\Omega)$, with
\begin{equation} \label{eq:Leibniz_lambda_gen}
\div(uF) = u^{\lambda} \div F + (F, Du)_{\lambda} \ \text{ on } \Omega,
\end{equation}
and 
\begin{equation} \label{eq:pairing_estimate_lambda}
|(F, Du)_\lambda| \le \|F\|_{L^\infty(\Omega; \R^n)} |Du| \ \text{ on } \Omega.
\end{equation}
\end{theorem}

The assumption $u^\lambda \in L^{1}(\Omega; |\div F|)$ in Theorem \ref{thm:pairlambda_Leibniz} is equivalent to the summability of any $\overline{\lambda}$-representative of $u$ with respect to the measure $|\div F|$, as showed in \cite[Lemma 3.2]{crasta2019pairings}. In particular, it is also equivalent to the summability of the majorant function $\mul$ given by \eqref{eq:upmlambda}. We collect this results in the proposition below.

\begin{proposition} \label{prop:summability_lambda}
Let $F\in \DM^{\infty}(\Omega)$ and $u\in BV(\Omega)$. Let $\lambda_1, \lambda_2 : \Omega \to [0, 1]$ be Borel functions. Then we have $u^{\lambda_1} \in L^{1}(\Omega; |\div F|)$ if and only if $u^{\lambda_2} \in L^{1}(\Omega; |\div F|)$. In addition, if $\lambda : \Omega \to [0, 1]$ is a Borel function, then $u^\lambda \in L^{1}(\Omega; |\div F|)$ if and only if $\mul \in L^{1}(\Omega; |\div F|)$.
\end{proposition}

\begin{proof}
The first equivalence is given by \cite[Lemma 3.2]{crasta2019pairings}. Then, if we assume $u^\lambda \in L^1(\Omega; |\div F|)$ for some $\lambda : \Omega \to [0, 1]$, by the first part of the statement, we have $u^+, u^- \in L^1(\Omega; |\div F|)$, since they correspond to the choices $\overline{\lambda} \equiv 1$ and $\overline{\lambda} \equiv 0$, respectively. Hence, we see that $\mul \in L^{1}(\Omega; |\div F|)$, in the light of \eqref{eq:upmlambda} and the fact that $|\div F| \ll \Haus{n-1}$. Finally, the second estimate in \eqref{eq:TuLambda_estimate} immediately gives the opposite implication.
\end{proof}

One of the main applications of the theory of divergence-measure fields has consisted in generalizing integration by parts formulas for rough domains and weakly differentiable functions. For the ease of the reader we collect in the following statements some versions of these results, suitably refined for the purposes of this paper. The first one is simply \cite[Lemma 3.1]{comi2017locally}.

\begin{lemma} \label{lem:div_comp_supp_no_trace}
Let $F \in \DM^{\infty}(\Omega)$ be such that ${\rm supp}(F) \Subset \Omega$. Then we have $\div F(\Omega) = 0$.
\end{lemma}

Now we deal with the integration on sets with finite perimeter. Such Gauss-Green formulas have been widely studied in the literature \cite{ChenComiTorres, CTZ, comi2022representation, comi2017locally, crasta2017anzellotti, crasta2019pairings, Silhavy1, Silhavy2}. However, we provide here a general result, allowing for sets which can touch the boundary of the definition domain $\Omega$, under minimal regularity assumptions. To this purpose, we recall that an open set $\Omega$ is said to be {\em weakly regular} if it is a bounded set with finite perimeter in $\R^n$ such that $\Haus{n - 1}(\partial \Omega) = \Haus{n - 1}(\redb \Omega)$, or, equivalently, $\Haus{n - 1}(\partial \Omega \setminus \redb \Omega) = 0$.

\begin{theorem} \label{thm:GG_app}
Let $F \in \DM^{\infty}(\Omega)$ and let $E \subseteq \Omega$ be of finite perimeter in $\R^n$. Assume that either $E \Subset \Omega$ or $\Omega$ is weakly regular. Then, there exist the {\em interior and exterior normal traces} of $F$ on $\redb E$; that is, the functions 
$${\rm Tr}^i(F, \redb E), {\rm Tr}^e(F, \redb E) \in L^\infty(\redb E; \Haus{n-1}),$$ 
which satisfy
\begin{equation}
 \begin{cases} \div F(E^1 \cap \Omega)  = - \displaystyle \int_{\redb E} {\rm Tr}^i(F, \redb E) \, d \Haus{n-1} \\
\|{\rm Tr}^i(F, \redb E)\|_{L^\infty(\redb E; \Haus{n-1})}  \le \|F\|_{L^\infty(E; \R^n)},  \end{cases} \label{eq:GG_1} \\
\end{equation}
and
\begin{equation}
\begin{cases} \div F( (E^1 \cup \redb E) \cap \Omega ) = - \displaystyle \int_{\redb E} {\rm Tr}^e(F, \redb E) \, d \Haus{n-1} \\
\|{\rm Tr}^e(F, \redb E)\|_{L^\infty(\redb E; \Haus{n-1})}  \le \begin{cases} \|F\|_{L^\infty(\Omega \setminus E; \R^n)} & \text{ if } E \subset \subset \Omega \\
\|F\|_{L^\infty(\Omega; \R^n)} & \text{ otherwise} \end{cases}. \end{cases} \label{eq:GG_2}
\end{equation}
In particular, if $\Omega$ is weakly regular, then there exists ${\rm Tr}^i(F, \partial \Omega) \in L^\infty(\partial \Omega; \Haus{n-1})$ satisfying 
\begin{equation}
\div F(\Omega) = - \int_{\partial \Omega} {\rm Tr}^i(F, \partial \Omega) \, d \Haus{n-1} \text{ and } \|{\rm Tr}^i(F, \partial \Omega)\|_{L^\infty(\partial \Omega; \Haus{n-1})} \le \|F\|_{L^\infty(\Omega; \R^n)}. \label{eq:GG_boundary_domain}
\end{equation}
%From a rigorous standpoint, the interior normal trace of $F$ on $\partial \Omega$ cannot be defined; however, since $\Omega$ is a bounded open set with Lipschitz boundary, then the zero extension $\hat{F}$ admits an interior normal trace ${\rm Tr}^i(\hat{F}, \partial \Omega)$, and such trace depends only on the values of $F$ inside $\Omega$ (for more details, we refer to \cite[Corollary 5.5, Remark 5.6]{comi2017locally}). Hence, with a little abuse of notation, we can write ${\rm Tr}^i(F, \partial \Omega)$, instead of ${\rm Tr}^i(\hat{F}, \partial \Omega)$.
\end{theorem}

\begin{proof}
For the proof of \eqref{eq:GG_1} and \eqref{eq:GG_2} in the case $E \Subset \Omega$ we refer to \cite[Theorem 3.2 and Corollary 3.6]{comi2017locally}. Then, if $\Omega$ is weakly regular, by \cite[Corollary 5.5 and Remark 5.6]{comi2017locally} we have that the zero extension of $F$ to $\R^{n}$, defined as
\begin{equation*} \hat{F}(x) := \begin{cases} F(x) & \text{if} \ x \in \Omega \\
0 & \text{if} \ x \in \R^{n} \setminus \Omega \end{cases}, \end{equation*}
satisfies $\hat{F} \in \DM^{\infty}(\R^{n})$, and there exist ${\rm Tr}^i(\hat{F}, \partial \Omega), {\rm Tr}^e(\hat{F}, \partial \Omega) \in L^\infty(\partial \Omega; \Haus{n-1})$ such that
\begin{equation} \label{eq:div_hat_F}
\div \hat{F} = \div F \res \Omega + {\rm Tr}^i(\hat{F}, \partial \Omega) \, \Haus{n-1} \res \partial \Omega \ \ \text{ and } \ \ {\rm Tr}^e(\hat{F}, \partial \Omega) = 0.
\end{equation}
Hence, if $E \subset \Omega$ is a set of finite perimeter in $\R^n$, we can apply \eqref{eq:GG_1} and \eqref{eq:GG_2} to $\hat{F}$ and obtain
\begin{align}
\div \hat{F}(E^1) & = - \int_{\redb E} {\rm Tr}^i(\hat{F}, \redb E) \, d \Haus{n-1},  \nonumber \\
\div \hat{F}(E^1 \cup \redb E) & = - \int_{\redb E} {\rm Tr}^e(\hat{F}, \redb E) \, d \Haus{n-1}. \label{eq:GG_2_hat_F}
\end{align}
Then, we notice that it must be $\Haus{n-1}(E^1 \setminus \Omega) = 0$, since $\Haus{n-1}(\Omega^1 \setminus \Omega) = 0$ (see \cite[Corollary 5.5]{comi2017locally}) and $E^1 \subset \Omega^1$. Therefore, we have 
$$\div \hat{F}(E^1) = \div \hat{F}(E^1 \cap \Omega) + \div\hat{F} (E^1 \setminus \Omega) = \div F (E^1 \cap \Omega),$$ 
since $|\div \hat{F}| \ll \Haus{n-1}$. Hence, given that 
$$\|{\rm Tr}^i(F, \redb E)\|_{L^\infty(\redb E; \Haus{n-1})} \le \|\hat{F}\|_{L^\infty(E; \R^n)} = \|F\|_{L^\infty(E; \R^n)},$$ we can set ${\rm Tr}^i(F, \redb E) := {\rm Tr}^i(\hat{F}, \redb E)$ to include the case $\Haus{n-1}(\partial^* E \cap \partial \Omega) > 0$, with a little abuse of notation, thus obtaining \eqref{eq:GG_1}. As for \eqref{eq:GG_2}, we exploit \eqref{eq:div_hat_F} and the fact that $\Haus{n-1}(E^1 \setminus \Omega) = 0$ to get
\begin{equation*}
\div \hat{F}(E^1 \cup \redb E) = \div F((E^1 \cup \redb E) \cap \Omega) + \int_{\redb E \cap \partial \Omega} {\rm Tr}^i(\hat{F}, \partial \Omega) \, d \Haus{n-1}.
\end{equation*}
In addition, we notice that $\nu_E(x) = \nu_\Omega(x)$ for $\Haus{n-1}$-a.e. $x \in \redb E \cap \partial \Omega$, since $E \subset \Omega$, so that, by \cite[Proposition 4.10]{comi2017locally}, we have 
\begin{equation*}
{\rm Tr}^i(\hat{F}, \redb E) = {\rm Tr}^i(\hat{F}, \partial \Omega) \ \text{ and } \ {\rm Tr}^e(\hat{F}, \redb E) = {\rm Tr}^e(\hat{F}, \partial \Omega) = 0 \ \ \Haus{n-1}\text{-a.e. on } \redb E \cap \partial \Omega,
\end{equation*} 
due to \eqref{eq:div_hat_F}.
Hence, exploiting \eqref{eq:GG_2_hat_F} we get
\begin{align*}
\div F((E^1 \cup \redb E) \cap \Omega)& = - \int_{\redb E} {\rm Tr}^e(\hat{F}, \redb E) \, d \Haus{n-1} - \int_{\redb E \cap \partial \Omega} {\rm Tr}^i(\hat{F}, \partial \Omega) \, d \Haus{n-1} \\
& = - \int_{\redb E \setminus \partial \Omega} {\rm Tr}^e(\hat{F}, \redb E) \, d \Haus{n-1} - \int_{\redb E \cap \partial \Omega} {\rm Tr}^i(\hat{F}, \redb E) \, d \Haus{n-1}.
\end{align*}
Hence, we obtain \eqref{eq:GG_2} by setting 
\begin{equation*}
{\rm Tr}^e(F, \redb E) := \begin{cases} {\rm Tr}^e(\hat{F}, \redb E) & \text{ on } \redb E \setminus \partial \Omega \\
{\rm Tr}^i(\hat{F}, \redb E) & \text{ on } \redb E \cap \partial \Omega \end{cases},
\end{equation*}
which clearly satisfies
\begin{align*}
\|{\rm Tr}^e(F, \redb E)\|_{L^\infty(\redb E; \Haus{n-1})} & \le \max \{ \|\hat{F}\|_{L^{\infty}(\R^n \setminus E; \R^n)}, \|\hat{F}\|_{L^{\infty}(E; \R^n)} \} = \|\hat{F}\|_{L^\infty(\R^n; \R^n)} \\
& = \|F\|_{L^\infty(\Omega; \R^n)} 
\end{align*} 
Finally, we obtain \eqref{eq:GG_boundary_domain} if we choose $E = \Omega$ in \eqref{eq:GG_1}.
\end{proof}

\begin{remark}
We point out that in Theorem \ref{thm:GG_app} we can equivalently assume that $E$ has finite perimeter only in $\Omega$. Indeed, if $E \Subset \Omega$, then $\Per(E) = \Per(E, \Omega)$. If instead $\Omega$ is weakly regular, then there exists a family of bounded open sets with smooth boundary $(\Omega_k)_{k \in \N}$ such that 
\begin{equation*}
\Omega_k \subset \Omega_{k+1}, \  \bigcup_{k = 0}^{+\infty} \Omega_k = \Omega \ \text{ and } \Per(\Omega_k) \to \Per(\Omega)
\end{equation*}
(see \cite[Theorem 1.1]{MR3314116} and the subsequent discussion).
Hence, given $E \subset \Omega$ with $\Per(E, \Omega) < + \infty$, it is clear that 
\begin{align*}
\Per(E \cap \Omega_k) = \Per(E \cap \Omega_k, \Omega) \le \Per(E, \Omega) + \Per(\Omega_k, \Omega) = \Per(E, \Omega) + \Per(\Omega_k)
\end{align*} 
Thus, by the lower semicontinuity of the perimeter we get
\begin{equation*}
\Per(E) \le \liminf_{k \to + \infty} \Per(E \cap \Omega_k) \le \Per(E, \Omega) + \lim_{k \to + \infty} \Per(\Omega_k) = \Per(E, \Omega) + \Per(\Omega).
\end{equation*}
Since clearly $\Per(E) \ge \Per(E, \Omega)$, we conclude that $E \subset \Omega$ has finite perimeter in $\Omega$ if and only if it has finite perimeter in $\R^n$.
\end{remark}

Exploiting the Leibniz rule given by Theorem \ref{thm:pairlambda_Leibniz}, we are able to provide an integration by parts formula up to the boundary of a Lipschitz domain for the divergence of a product between a divergence-measure field and a suitable scalar function of bounded variation. Although similar results are already present in the literature (see for instance \cite{Anzellotti_1983, ChenComiTorres, CF1, comi2022representation, comi2017locally, crasta2017extension, Frid2}), to the best of our knowledge it has never been formulated with this level of generality before.

\begin{theorem} \label{thm:GG_boundary_domain_u}
Let $\Omega$ be an open bounded set with Lipschitz boundary.
Let $F \in \DM^{\infty}(\Omega)$ and $u \in BV(\Omega)$ be such that $u^\lambda \in L^1(\Omega; |\div F|)$ for some Borel function $\lambda : \Omega \to [0, 1]$. Then we have
\begin{equation} \label{eq:GG_boundary_domain_u}
\int_{\Omega} u^\lambda \, d \div F + (F, Du)_\lambda(\Omega) = \div(uF)(\Omega) = - \int_{\partial \Omega} {\rm Tr}_{\partial \Omega}(u) {\rm Tr}^i(F, \partial \Omega) \, d \Haus{n-1}.
\end{equation}
\end{theorem}

\begin{proof}
We notice that $\Omega$ is a weakly regular set. Hence, arguing as in the proof of Theorem \ref{thm:GG_app}, we consider the zero extension $\hat{F}$ of $F$ to $\R^{n}$. In addition, by \cite[Theorem 3.87]{AFP} we know that ${\rm Tr}_{\partial \Omega}(u) \in L^1(\partial \Omega; \Haus{n-1})$ and that the zero extension $\hat{u}$ of $u$ to $\R^n$ satisfies $\hat{u} \in BV(\R^n)$ and $D\hat{u} = Du$ on $\Omega$. For $N > 0$, we consider the truncation $T_N(\hat{u})$, which clearly satisfies $T_N(\hat{u}) = \widehat{T_N(u)}$. Thanks to \cite[Proposition 3.1]{crasta2017anzellotti}, we know that 
\begin{equation*}
{\rm Tr}^i(T_N(\hat{u}) \hat{F}, \partial \Omega) = {\rm Tr}_{\partial \Omega}(T_N(u)) {\rm Tr}^i(\hat{F}, \partial \Omega) \ \Haus{n-1}\text{-a.e. on } \partial \Omega.
\end{equation*}
By applying \eqref{eq:GG_1} to the set $E = \Omega$ in the domain $\R^n$, we get
\begin{equation*}
\div (\widehat{T_N(u)} \hat{F})(\Omega^1) =  - \int_{\partial \Omega} {\rm Tr}_{\partial \Omega}(T_N(u)) {\rm Tr}^i(\hat{F}, \partial \Omega) \, d \Haus{n-1}.
\end{equation*}
However, $\Omega$ is weakly regular, which implies that $\Haus{n-1}(\Omega^1 \setminus \Omega) = 0$ (see \cite[Corollary 5.5]{comi2017locally}), and so 
$$\div (\widehat{T_N(u)} \hat{F})(\Omega^1) = \div (\widehat{T_N(u)} \hat{F})(\Omega) = \div (T_N(u) F)(\Omega),$$
since $|\div (\widehat{T_N(u)} \hat{F})| \ll \Haus{n-1}$ and $\widehat{T_N(u)} \hat{F} = T_N(u) F$ on $\Omega$. Hence, by \eqref{eq:Leibniz_lambda_gen}, we get
\begin{equation} \label{eq:GG_TNu}
\int_{\Omega} T_N(u)^\lambda \, d \div F + (F, DT_N(u))_\lambda(\Omega) = \div(T_N(u)F)(\Omega) = - \int_{\partial \Omega} {\rm Tr}_{\partial \Omega}(T_N(u)) {\rm Tr}^i(F, \partial \Omega) \, d \Haus{n-1},
\end{equation}
where we set ${\rm Tr}^i(F, \partial \Omega) := {\rm Tr}^i(\hat{F}, \partial \Omega)$, as in the proof of Theorem \ref{thm:GG_app}. Since $T_N(u) \to u$ in $BV(\Omega)$, and therefore in $BV(\Omega)$-strict as $N \to + \infty$, \cite[Theorem 3.88]{AFP} implies that ${\rm Tr}_{\partial \Omega}(T_N(u)) \to {\rm Tr}_{\partial \Omega}(u)$ in $L^1(\partial \Omega; \Haus{n-1})$ as $N \to + \infty$. As for the left hand side, combining Proposition \ref{prop:mul} and Proposition \ref{prop:summability_lambda}, we exploit Lebesgue's dominated convergence theorem to obtain $T_N(u)^\lambda \to u^\lambda$ in $L^1(\Omega; |\div F|)$ as $N \to + \infty$. Then, we know that $(F, DT_N(u))_\lambda \weakto (F, Du)_\lambda$ in $\Mm(\Omega)$ by \cite[Remark 4.5]{crasta2019pairings}. Let $\Omega_\delta = \{ x \in \Omega : {\rm dist}(x, \partial \Omega) > \delta\}$ for $\delta > 0$ such that $\Omega_\delta \neq \emptyset$. By \eqref{eq:pairing_estimate_lambda}, it is clear that the sequence of measures $\left ( |(F, DT_N(u))_\lambda| \right)_N$ is uniformly bounded, so that there exists some measure $\gamma \ge |(F, Du)_\lambda|$ such that $|(F, DT_N(u))_\lambda| \weakto \gamma$ in $\Mm(\Omega)$, up to a subsequence. Due to the non-concentration property of Radon measures, we know that $\gamma(\partial \Omega_\delta) = 0$ for $\Leb{1}$-a.e. $\delta >0$. We fix a sequence $(\delta_k)_{k \in \N}, \delta_k \to 0^+$, of such good values of $\delta > 0$. Thus, by \cite[Proposition 1.62]{AFP}, possibly up to a subsequence we get
\begin{equation} \label{eq:conv_pair_Omega_k}
(F, DT_N(u))_\lambda(\Omega_{\delta_k}) \to (F, Du)_\lambda(\Omega_{\delta_k}) \ \text{ as } N \to + \infty.
\end{equation}
On the other hand, $|D T_N(u)| \le |Du|$ and \eqref{eq:pairing_estimate_lambda} yield
\begin{equation*}
|(F, DT_N(u))_\lambda|(\Omega \setminus \Omega_{\delta_k}) \le \|F\|_{L^\infty(\Omega; \R^n)} |Du|(\Omega \setminus \Omega_{\delta_k}) \to 0 \ \text{ as } k \to + \infty.
\end{equation*}
Hence, for all $\eps > 0$, there exists $k_0 \in \N$ such that 
\begin{equation*}
\max\{ |(F, D u)_\lambda|(\Omega \setminus \Omega_{\delta_k}), |(F, DT_N(u))_\lambda|(\Omega \setminus \Omega_{\delta_k})\} \le \|F\|_{L^\infty(\Omega; \R^n)} |Du|(\Omega \setminus \Omega_{\delta_k}) < \eps
\end{equation*}
for all $k \ge k_0$ and $N > 0$. 
Thus, thanks to \eqref{eq:conv_pair_Omega_k} we obtain
\begin{align*}
\limsup_{N \to + \infty} \left | (F, DT_N(u))_\lambda(\Omega) - (F, D u)_\lambda(\Omega) \right | & \le \lim_{N \to + \infty} \left | (F, DT_N(u))_\lambda(\Omega_{\delta_k}) - (F, Du)_\lambda(\Omega_{\delta_k}) \right | \\
& + \limsup_{N \to + \infty} |(F, DT_N(u))_\lambda|(\Omega \setminus \Omega_{\delta_k}) + |(F, D u)_\lambda|(\Omega \setminus \Omega_{\delta_k}) \\
& \le 2 \eps,
\end{align*}
which, since $\eps$ is arbitrary, implies 
\begin{equation*}
\lim_{N \to + \infty} (F, DT_N(u))_\lambda(\Omega) = (F, D u)_\lambda(\Omega).
\end{equation*}
All in all, we pass to the limit as $N \to + \infty$ in \eqref{eq:GG_TNu}, finally obtaining \eqref{eq:GG_boundary_domain_u}.
\end{proof}

\subsection{Measures in the dual of $BV$} We recall the definition of the dual of the space $BV$.

\begin{definition}
We denote by $BV(\Omega)^*$ the dual of the space $BV(\Omega)$; that is, the space of linear functionals $\T: BV(\Omega) \to \R$ for which there exists a constant $C > 0$ such that 
\begin{equation*}
|\T(u)| \le C \|u\|_{BV(\Omega)} \text{ for all } u \in BV(\Omega).
\end{equation*}
\end{definition}

It is well known that there are some elements in the dual of $BV$ whose action can be represented as the integration of a suitable representative of the $BV$ function against a Radon measure (see for instance \cite{Phuc_Torres}). Indeed, given $\mu \in \MH(\Omega)$, the linear functional 
\begin{equation*} 
%\label{eq:T_\mu_classic}
\T_{\mu}(u) = \int_\Omega u^* \, d \mu  \quad \text{ for } u \in BV(\Omega) \cap L^{\infty}(\Omega)
\end{equation*}
is well defined, although not necessarily continuous. We point out that the assumption $\mu \in \MH(\Omega)$ is necessary, since otherwise $u^*(x)$ might not be well defined for $|\mu|$-a.e. $x \in \Omega$, and that the choice of the precise representative $u^*$ ensures the linearity of $\T_{\mu}$, due to \eqref{eq:u_star_mean_lim}. Therefore, we choose the following notation:

%Indeed, due to Theorem \ref{thm:Anzellotti_Giaquinta} below, for all $u \in BV(\Omega) \cap L^\infty(\Omega)$ there exists a sequence $(u_k)_{k \in \N} \subset C^\infty(\Omega) \cap BV(\Omega) \cap L^\infty(\Omega)$, such that $u_k \to u$ in $BV(\Omega)$-strict, $u_k(x) \to u^*(x)$ for $\Haus{n-1}$-a.e. $x \in \Omega$ and $\|u_k\|_{L^\infty(\Omega)} \le 2 \|u\|_{L^\infty(\Omega)}$, and so 
%\begin{equation*}
%\lim_{k \to + \infty} \T_{\mu}(u_k) = \lim_{k \to + \infty} \int_\Omega u_k \, d \mu = \int_\Omega u^* \, d \mu
%\end{equation*}
%by Lebesgue's dominated convergence theorem.

\smallskip

\begin{minipage}{1\linewidth}
\begin{center} \textit{given $\mu \in \MH(\Omega)$, we say that $\mu \in BV(\Omega)^*$ if there exists $\T \in BV(\Omega)^*$ such that} \end{center}
$$\T(u) = \T_{\mu}(u) \quad \textit{ for all } \ u\in BV(\Omega) \cap L^\infty(\Omega).$$
\end{minipage}

\smallskip

Under the additional assumption that $\Omega$ is an open bounded set with Lipschitz boundary, it was proved in \cite[Theorem 8.2]{Phuc_Torres} that $\T_{\mu}$ is continuous on $BV(\Omega)\cap L^{\infty}(\Omega)$ with respect to the $BV$-norm if and only if there exists $C > 0$ such that 
\begin{equation*}
|\mu(U \cap \Omega )| \le C\, \Per(U) \quad \text{ for all open sets } U \subset \R^n \text{ with smooth boundary.}
\end{equation*}
We further explore this and similar results in Section \ref{sec:admissible_measures}.

Given a pair $(\mu, \lambda) \in \MH(\Omega) \times \B_{b}(\Omega)$, it is also possible to define 
\begin{equation} \label{def:T_mu_lambda}
\T_{\mu, \lambda}(u) = \int_\Omega u^\lambda \, d \mu  \quad \text{ for } u \in BV(\Omega) \cap L^{\infty}(\Omega).
\end{equation}
However, unless $\lambda(x) \equiv \frac{1}{2}$ for $|\mu|$-a.e. $x \in \Omega$ or $|\mu|(\Sigma) = 0$ for all $\Haus{n-1}$-rectifiable sets $\Sigma$, this mapping is not linear. Indeed, arguing analogously as in \cite[Remark 4.6]{crasta2019pairings}, we notice that for any $u \in BV(\Omega) \cap L^\infty(\Omega)$ we have
\begin{equation*}
\T_{\mu, \lambda}(u) + \T_{\mu, \lambda}(- u) = \int_\Omega \left ( u^\lambda + (-u)^\lambda \right ) \, d \mu = \int_\Omega (2 \lambda - 1) (u^+ - u^-) \, d \mu. 
\end{equation*}
Nevertheless, it is worth remarking that for a relevant class of measures the functional $\T_{\mu, \lambda}$ can be extended to a continuous functional defined on the whole space $BV(\Omega)$, see Definition \ref{def:muadmissible} and Lemma \ref{lem:lambda_int_approx}.
\medskip

\section{$\lambda$-approximation}

It is natural to expect that many properties of the $\lambda$-pairing $(F,Du)_{\lambda}$ can be obtained as direct extensions of properties satisfied by the pairing when $u$ is a smooth function. The goal of this section is to provide the required approximation result, Theorem \ref{lemma:pairlambda-vs-area}, which in turn is based on Theorem \ref{thm:smooth_lambda_approx}, a finer version of Anzellotti-Giaquinta's approximation that is specifically designed to enforce also the $\Haus{n-1}$-a.e. approximation of the $\lambda$-representative $u^{\lambda}$.

We start by stating a first refinement of Anzellotti-Giaquinta approximation theorem for $BV$ functions involving also the $\Haus{n-1}$-a.e. pointwise convergence and the strict convergence with respect to the area functional. As customary, we say that any radial function $\rho \in C^{\infty}_{c}(B_{1})$ such that $\rho \ge 0$ and $\displaystyle \int_{B_{1}} \rho \, dx = 1$ is a {\em standard mollifier}. For all $\eps > 0$ we set $\rho_\eps(x) = \eps^{-n} \rho\left (\frac{x}{\eps} \right )$, and we recall that, for all $u \in BV(\Omega)$, by \cite[Corollary 3.80]{AFP} we have 
\begin{equation} \label{eq:u_star_moll_point_lim}
\lim_{\eps \to 0^+} (\rho_\eps \ast u)(x) = u^*(x) \ \text{ for } \Haus{n-1}\text{-a.e. } x \in \Omega.
\end{equation}

\begin{theorem} \label{thm:Anzellotti_Giaquinta} 
Let $u \in BV(\Omega)$. Then there exists $(u_{\eps})_{\eps > 0} \subset BV(\Omega) \cap C^{\infty}(\Omega)$ such that 
\begin{enumerate}
	\item $u_{\eps} \to u$ in $L^{1}(\Omega)$ as $\eps \to 0$,
	\item $|Du_{\eps}|(\Omega) \le |Du|(\Omega) + 4 \eps$ for all $\eps > 0$, and so $\displaystyle \lim_{\eps \to 0} |Du_{\eps}|(\Omega) = |Du|(\Omega)$,
	\item $\sqrt{1 + |Du_{\eps}|^2}(\Omega) \le \sqrt{1 + |Du|^2}(\Omega) + 4 \eps$ for all $\eps > 0$, and so $\displaystyle \lim_{\eps \to 0} \sqrt{1 + |Du_{\eps}|^2}(\Omega) = \sqrt{1 + |Du|^2}(\Omega)$,
	\item $u_{\eps}(x) \to u^{*}(x)$ for $\Haus{n - 1}$-a.e. $x \in \Omega$ as $\eps \to 0$,
	\item $\displaystyle \|u_{\eps}\|_{L^{\infty}(\Omega)} \le (1 + \eps) \|u\|_{L^{\infty}(\Omega)}$ for all $\eps > 0$,
	\item if $\Omega$ is an open set with bounded Lipschitz boundary, then $\Tr_{\partial \Omega}(u_\eps) = \Tr_{\partial \Omega}(u)$ for all $\eps >0$.
\end{enumerate}
\end{theorem}

\begin{proof}
The proof is based on a slight modification of the construction of the approximating sequence in Anzellotti-Giaquinta theorem, for which we refer for instance to \cite[Theorem 3, Section 5.2.2]{evans2015measure}.

Fix $\eps > 0$. Given a positive integer $m$, we set $\Omega_{0} = \emptyset$, define for each $k \in \mathbb{N}, k \ge 1$ the sets
\[ \Omega_{k} = \left \{ x \in \Omega : \ \mathrm{dist}(x, \partial \Omega) > \frac{1}{m + k} \right \} \cap B_{k + m} \]
and then we choose $m$ such that $|Du|(\Omega \setminus \Omega_{1}) < \eps$. 

We define now $\Sigma_{k} := \Omega_{k+1} \setminus \overline{\Omega_{k-1}}$. Since $(\Sigma_{k})_{k \ge 1}$ is an open cover of $\Omega$, then there exists a partition of unity subordinate to that open cover; that is, a sequence of functions $(\zeta_{k})_{k \ge 1}$ such that:
\begin{enumerate}
	\item $\zeta_{k} \in C^{\infty}_{c}(\Sigma_{k})$;
	\item $0 \le \zeta_{k} \le 1$;
	\item $\sum_{k = 1}^{+\infty} \zeta_{k} = 1$ on $\Omega$.
\end{enumerate}

Next, we take a standard mollifier $\rho$ and for all $k \in \N, k \ge 1$ we choose $\delta_{k} = \delta_{k}(\eps) > 0$ small enough such that the following conditions hold
\begin{align}
\mathrm{supp}(\rho_{\delta_{k}} \ast \zeta_{k}) & \subset \Sigma_{k}, \label{eq:supp_conv_cutoff} \\
  \|\rho_{\delta_{k}} \ast (u\zeta_{k}) - u\zeta_{k}\|_{L^{1}(\Omega)} & < \frac{\eps}{2^{k}}, \label{eq:conv_cutoff_L1} \\
\|\rho_{\delta_{k}} \ast (u\nabla\zeta_{k}) - u\nabla\zeta_{k}\|_{L^{1}(\Omega; \mathbb{R}^{n})} & < \frac{\eps}{2^{k}}, \label{eq:conv_cutoff_nabla} \\
\| \rho_{\delta_{k}} \ast \zeta_{k} - \zeta_{k} \|_{L^{\infty}(\Omega)} & < \frac{\eps}{2^{k}} \label{eq:conv_cutoff_infty}.
\end{align}
Then we define $\displaystyle u_{\eps} := \sum_{k=1}^{+\infty} \rho_{\delta_{k}} \ast (u\zeta_{k})$. 
We notice that, for any fixed $x \in \Omega$, there exists a unique $k = k(x) \ge 1$ such that $x \in \Sigma_k \cap \Sigma_{k+1}$, so that, by \eqref{eq:supp_conv_cutoff},
\begin{equation} \label{eq:sum_locally_finite_terms}
u_{\eps}(x) = \sum_{j = 0}^{1} (\rho_{\delta_{k + j}} \ast (u \zeta_{k + j}))(x).
\end{equation}
Hence, $u_{\eps} \in C^{\infty}(\Omega)$, since locally there are at most two nonzero terms in the sum. 
The proof of points (1) and (2) follows in a standard way from \eqref{eq:supp_conv_cutoff}, \eqref{eq:conv_cutoff_L1} and \eqref{eq:conv_cutoff_nabla} (see for instance \cite[Theorem 3, Section 5.2.2]{evans2015measure} and \cite[Theorem 1.17]{giusti1984minimal}). Arguing in a similar way, for any $(\varphi, \eta) \in C^1_c(\Omega; \R^n \times \R)$ such that $\|(\varphi, \eta)\|_{L^{\infty}(\Omega; \R^n \times \R)} \le 1$ we obtain
\begin{align*}
\int_\Omega \eta + u_\eps \div \varphi \, dx & = \int_\Omega \eta + u \div(\zeta_1 (\rho_{\delta_1} * \varphi)) \, dx + \sum_{k = 2}^{+\infty} \int_\Omega u \div(\zeta_k (\rho_{\delta_k} * \varphi)) \, dx + \\
& - \sum_{k = 1}^{+\infty} \int_\Omega \varphi \cdot \left ( \rho_{\delta_k} * \left (u \nabla \zeta_k \right ) - u \nabla \zeta_k \right ) \, dx \\
& \le \sqrt{1 + |Du|^2}(\Omega) + \sum_{k = 2}^{+ \infty} |Du|(\Sigma_k) + \sum_{k = 1}^{+ \infty} \|\rho_{\delta_{k}} \ast (u\nabla\zeta_{k}) - u\nabla\zeta_{k}\|_{L^{1}(\Omega; \mathbb{R}^{n})} \\
& \le \sqrt{1 + |Du|^2}(\Omega) + 2 |Du|(\Omega \setminus \Omega_1) + \sum_{k = 1}^{+ \infty} \frac{\eps}{2^k} \\
& \le \sqrt{1 + |Du|^2}(\Omega) + 3 \eps.
\end{align*}
Together with the lower semicontinuity of the area functional with respect to the $L^1$ convergence, this implies point (3).
Then, since $\zeta_{k}$ is continuous for every $k$ and $u \in BV(\Omega)$, by \eqref{eq:u_star_moll_point_lim} we have
\begin{equation*}
\lim_{\eps \to 0^+} (\rho_{\delta_{k}} \ast (u \zeta_{k}))(x) = \zeta_{k}(x) u^*(x) \ \text{ for } \Haus{n - 1}\text{-a.e. } x \in \Omega, \end{equation*}
where the $\Haus{n - 1}$-negligible set of points depends only on $u$. Hence, for $\Haus{n-1}$-a.e. $x \in \Omega$, by \eqref{eq:sum_locally_finite_terms} we get
\begin{equation*}
\lim_{\eps \to 0^+} u_{\eps}(x) = \lim_{\eps \to 0^+} \sum_{j = 0}^{1} (\rho_{\delta_{k + j}} \ast (u \zeta_{k + j}))(x) = u^{*}(x) \sum_{j = 0}^{1} \zeta_{k + j}(x) = u^{*}(x).
\end{equation*}
As for point (5), we can assume without loss of generality that $u \in L^{\infty}(\Omega)$. For any $x \in \Omega$ there exists a unique $k \ge 1$ such that $\sum_{j = 0}^1 \zeta_{k+j}(x) = 1$. Therefore we apply \eqref{eq:conv_cutoff_infty} and \eqref{eq:sum_locally_finite_terms} to get
\begin{align*}
|u_{\eps}(x)| & \le  \|u\|_{L^{\infty}(\Omega)} \sum_{j = 0}^{1} (\rho_{\delta_{k + j}} \ast \zeta_{k + j}))(x) \le \|u\|_{L^{\infty}(\Omega)} \left ( 1 + \sum_{j = 0}^{1} |(\rho_{\delta_{k + j}} \ast \zeta_{k + j})(x) - \zeta_{k + j}(x)| \right) \\
& \le (1 + \eps) \|u\|_{L^{\infty}(\Omega)}.
\end{align*}
Finally, if $\Omega$ is an open set with bounded Lipschitz boundary, then the trace operator is well defined and, thanks to \cite[Theorem 3.87]{AFP}, it satisfies
\begin{equation*}
\lim_{\rho \to 0} \frac{1}{\rho^n} \int_{\Omega \cap B_\rho(x)} |u(y) - \Tr_{\partial \Omega}(u)(x)| \, d y = 0 \text{ for } \Haus{n-1}\text{-a.e. } x \in \partial \Omega.
\end{equation*}
Hence, arguing as in \cite[Remarks 1.18 and 2.12]{giusti1984minimal}, we conclude that (6) holds true.
\end{proof}

We exploit Theorem \ref{thm:Anzellotti_Giaquinta} in order to obtain the following finer $\lambda$-approximation theorem, which is then applied to the approximation of $\lambda$-pairings (Theorem \ref{lemma:pairlambda-vs-area} below) and is exploited at length in Section \ref{sec:admissible_measures}. In addition, it represents one of the key tools in the proofs of the main results of \cite{LeoComi}.

\begin{theorem} [$\lambda$-approximation] \label{thm:smooth_lambda_approx}
Let $\lambda:\Omega\to [0,1]$ be a given Borel function. Then for every $u\in BV(\Omega)$ there exists a sequence $(u^\lambda_k)_{k \in \N}\subset C^\infty(\Omega) \cap BV(\Omega)$ such that we have the following:
\begin{enumerate}
\item $u^\lambda_k \to u$ in $BV(\Omega)$-strict as $k \to + \infty$,
\item $u^\lambda_k(x) \to u^\lambda(x)$ for $\Haus{n-1}$-a.e. $x \in \Omega$ as $k\to +\infty$,
\item $\displaystyle \lim_{k \to + \infty} \sqrt{1 + |Du^{\lambda}_k|^2}(\Omega) = \sqrt{1 + |Du|^2}(\Omega)$,
\item $\|u^\lambda_k\|_{L^\infty(\Omega)} \le \left( 1 + \frac{1}{k} \right) \|u\|_{L^{\infty}(\Omega)}$ for all $k \in \N$,
\item if $\Omega$ is an open set with bounded Lipschitz boundary, then $\Tr_{\partial \Omega}(u^\lambda_k)(x) = \Tr_{\partial \Omega}(u)(x)$ for $\Haus{n-1}$-a.e. $x \in \partial \Omega$ and all $k \in \N$.
\end{enumerate}
%In addition, if $u\in BV(\Omega)\cap L^\infty(\Omega)$, then we have $|u^\lambda_k(x)| \le \left( 1 + \frac{1}{k} \right) \|u\|_{L^{\infty}(\Omega)}$ for all $x \in \Omega$ and $k \in \N$. 
\end{theorem}
\begin{proof}
We first prove the theorem under the assumption $\|u\|_{L^{\infty}(\Omega)} < +\infty$. It is enough to construct a sequence $(\ulk)_k$ of smooth functions that converge to $u$ in $BV(\Omega)$-strict, to $u^* =\widetilde{u}= \ul$ $\Haus{n-1}$-a.e. on $\Omega\setminus J_u$, and to $\ul$ locally in measure with respect to $\mu = \Haus{n-1}\restrict J_u$. Then, we conclude by extracting a suitable subsequence that converges $\mu$-almost everywhere to $\ul$, hence $\Haus{n-1}$-a.e. on $\Omega$.  

Let $(u_\eps)_{\eps > 0}$ be the smooth approximation of $u$ given by Theorem \ref{thm:Anzellotti_Giaquinta}. We choose a sequence $\eps = \frac{1}{k}$, and we set $u_k := u_{\frac{1}{k}}$ for simplicity. Since $u_k(x) \to u^*(x)$ for $\Haus{n-1}$-a.e. $x\in \Omega$, the idea is to define $\ulk$ as a suitable perturbation of $u_k$ near the jump set $J_u$, and then show that $\ulk$ satisfies the convergence in measure stated above. The proof will be split into some steps.
\medskip

\textit{Step one: local construction and estimates.} 

We fix $\eps>0$ and a Borel set $A\Subset \Omega$, then we consider the set 
\[
S = S(u,A,\eps) = \left\{x\in A\cap J_{u}:\ u^+(x)-u^-(x)>\eps\right\}\,.
\]
Notice that $\Haus{n-1}(S)<+\infty$, hence $\mu(S)<+\infty$. Up to choosing the parameter $m$ in the proof of Theorem \ref{thm:Anzellotti_Giaquinta} to be large enough, we can assume that $u_k = u\ast \rho_k$ on $A$, where $\rho_k(|x|) = k^n \rho(k|x|)$ is a standard mollifier with support in $B_{1/k}$, for $k$ sufficiently large. For $\Haus{n-1}$-a.e. $x\in S$, we define the \textit{blow-up of $u$ at $x$} as the step function $u_{x, \infty}:\R^{n}\to \R$ defined by
\[
u_{x, \infty}(y) = 
\begin{cases}
u^{+}(x) & \text{if }(y-x)\cdot \nu_u(x) \ge 0,\\
u^{-}(x) & \text{otherwise}.
\end{cases}
\] 
Therefore we have
\[
\lim_{r\to 0} \frac{1}{r^{n}} \int_{B_r(x)}|u(y) - u_{x, \infty}(y)|\, dy = 0\,.
\]
For $\Haus{n-1}$-a.e. $x\in S$ we define the Borel function $\tau = \tau_k(x)\in [-1,1]$ as the unique implicit solution of 
\begin{equation}\label{eq:tauprop}
u_{x, \infty}\ast \rho_k(x + \tau\nu_u(x)/k) = u^\lambda(x)\,.
\end{equation}
Note that $\tau_k(x)$ can be written as the composition of a continuous function, depending only on the mollifier $\rho_k$, with the function $\lambda(x)$.

Fix $\delta \in (0,1)$ to be later chosen, and consider the following perturbation of $u_k$ inside a ball $B = B_r(x)$ centered at $x\in S$:
\begin{equation} \label{def:u_tau_B}
u^{\tau,B}_k(y) := \phi_{r,\delta}(y-x)u_k(y+\tau\nu_u(x)/k)  + \big(1-\phi_{r,\delta}(y-x)\big) u_k(y)\,,
\end{equation}
where $\tau\in [-1,1]$ and $\phi_{r,\delta}(z)\in [0,1]$ is a radially symmetric cut-off function of class $C^\infty$, with compact support in $B_r$ and such that $\phi_{r,\delta}(z) = 1$ if $|z|<(1-\delta)r$ and $\|\nabla\phi_{r,\delta}\|_{L^{\infty}(B_r; \R^n)} < \frac{2}{\delta r}$. Clearly, $u^{\tau,B}_k$ is a smooth function obtained by locally ``gluing'' a suitable translation of $u_k$ with $u_k$ itself, it coincides with $u_k$ outside $B$, and satisfies
\begin{equation}\label{eq:utaukappainx}
u^{\tau,B}_k(x) = u_k(x+\tau\nu_u(x)/k) = u\ast \rho_{k}(x+\tau\nu_u(x)/k)\,.
\end{equation}
Its gradient is given by
\[
\nabla u^{\tau,B}_k(y) = \nabla u_{k}(y) + \nabla\phi_{r,\delta}(y-x)\big(u_{k}(y+v_{k}) - u_{k}(y)\big) + \phi_{r,\delta}(y-x)\big(\nabla u_{k}(y+v_{k}) - \nabla u_{k}(y)\big)\,.
\]
where we have set $v_{k} = \frac{\tau}{k}\nu_u(x)$. For the sake of simplicity, we also set $C_{\delta} = B\setminus B_{(1-\delta)r}(x)$. We find that
\begin{align}\nonumber
\int_B |\nabla u^{\tau,B}_k(y)| \, dy &\le 
\int_{B_{(1-\delta)r}(x)} |\nabla u_k(y+v_{k})|\, dy + \int_{C_{\delta}} |\nabla u_{k}(y)| \, dy + \frac{2}{\delta r}\int_{C_{\delta}}|u_{k}(y+v_{k}) - u_{k}(y)|\, dy \\\label{eq:stimashift1}
&\qquad\qquad + \int_{C_{\delta}} |\nabla u_{k}(y+v_{k}) - \nabla u_{k}(y)|\, dy\,.
\end{align}
Then, owing to the fact that $|Du_{k}|(A) \le |Du|(A + B_{1/k})$, the first term in the right-hand side of \eqref{eq:stimashift1} can be estimated as follows:
\begin{equation}\label{eq:stimanablauBdk}
\int_{B_{(1-\delta)r}(x)} |\nabla u_k(y+v_{k})|\, dy \le |Du|(B_{\delta,k})\,,
\end{equation}
where we have set 
\[
B_{\delta,k} = B_{(1-\delta)r}(x) + B_{2/k} = B_{(1-\delta)r + 2/k}(x)\,.
\]
On observing that $|v_{k}|\le k^{-1}$, the third term in the right-hand side of \eqref{eq:stimashift1} can be estimated as follows:
\begin{equation}\label{eq:stimauCdelta}
\frac{2}{\delta r}\int_{C_{\delta}}|u_{k}(y+v_{k}) - u_{k}(y)|\, dy \le \frac{2}{k\delta r} |Du|(C_{\delta,k})\,,
\end{equation}
where $C_{\delta,k} = C_{\delta} + B_{2/k}$. 
Concerning the fourth term in the right-hand side of \eqref{eq:stimashift1}, we set $\rho_{k,x}(z) = \rho_k(z-v_{k})$ and denote by $u_{k,z}$ the average of $u$ on $B_{2/k}(z)$, so that we obtain
\begin{align*}
\int_{C_{\delta}} |\nabla u_{k}(y+v_{k}) - \nabla u_k(y)|\, dy &\le \int_{C_{\delta}} \left|\int_{\Omega} u(y)\Big(\nabla\rho_{k,x}(z-y) - \nabla \rho_k(z-y)\Big) \, dy\right|\, dz\\
&= \int_{C_{\delta}} \left|\int_{B_{2/k}(z)} (u(y) - u_{k,z})\Big(\nabla\rho_{k,x}(z -y) - \nabla \rho_k(z-y)\Big)\, dy\right| dz \\
&\le \frac{\|\nabla^2\rho_{k}\|_{L^{\infty}(B_1; \R^{n^2})}}{k} \int_{C_{\delta}} \int_{B_{2/k}(z)} |u(y) - u_{k,z}|\, dy\, dz\\
&\le Ck^{n+1} \int_{C_{\delta}} \int_{B_{2/k}(z)} |u(y) - u_{k,z}|\, dy\, dz\\
&\le Ck^{n} \int_{C_{\delta}} |Du|\big(B_{2/k}(z)\big)\, dz \,,
\end{align*}
where we have used the fact that $\|\nabla^2\rho_{k}\|_{L^{\infty}(B_1; \R^{n^2})} \le Ck^{n+2}$ and, in the last step, the Poincar\`e-Wirtinger inequality on $B_{2/k}(z)$ (note that in this last estimate, as well as in the next ones, we will denote by $C$ a dimensional constant that can possibly change from one line to another). We can push further the estimate by noticing that 
\begin{align*}
\int_{C_{\delta}} |Du|\big(B_{2/k}(z)\big)\, dz &= \int_{C_{\delta}}\int_{B_{2/k}(z)} \, d|Du|(y)\, dz\\
&= \int_{C_{\delta} + B_{2/k}}\int_{C_{\delta}\cap B_{2/k}(y)} \, dz\, d|Du|(y)\\
&= \int_{C_{\delta} + B_{2/k}} |C_{\delta}\cap B_{2/k}(y)|\, d|Du|(y)\\
&\le \frac{C}{k^n} |Du|\big(C_{\delta} + B_{2/k}\big) = \frac{C}{k^n} |Du|\big(C_{\delta,k}\big)\,.
\end{align*}
This leads us to
\begin{equation}\label{eq:stimashift2}
\int_{C_{\delta}} |\nabla u_k(y+v_{k}) - \nabla u_k(y)|\, dy \le C|Du|\big(C_{\delta,k}\big)\,.
\end{equation}
Consequently, if we plug \eqref{eq:stimanablauBdk}, \eqref{eq:stimauCdelta}, \eqref{eq:stimashift2} into  \eqref{eq:stimashift1}, we obtain
\begin{equation}\label{eq:mainlocalest}
\int_B |\nabla u^{\tau,B}_k| \, dy \le 
|D u|(B_{\delta,k}) + \left(C + \frac{2}{k\delta r}\right) |Du|\big(C_{\delta,k}\big)\,,
\end{equation}
where $C$ is a constant only depending on the dimension $n$. Note that as soon as $k\ge (r\delta)^{-1}$ the inequality improves to
\begin{equation} \label{eq:mainocaltest_1}
\int_B |\nabla u^{\tau,B}_k| \, dy \le 
|D u|(B_{\delta,k}) + (C+2) |Du|\big(C_{\delta,k}\big)\,.
\end{equation}
\medskip

\textit{Step two: from local to global.}

Here we show how to use the local construction and the estimate \eqref{eq:mainlocalest} provided by Step one, to define a sequence of approximations that behaves well on a compact subset of $S$ with large measure. To this aim we must first guarantee the continuity and the uniformity of some quantities that appear in Step one, and then apply the appropriate covering theorem. 

By Lusin and Egoroff Theorems, for every $\eta \in (0, 1)$ we can find a compact set $S_{\eta}\subset S$ satisfying the following properties:
\begin{itemize}
\item[(i)] $\Haus{n-1}(S_{\eta}) \ge (1-\eta)\Haus{n-1}(S)$;
\item[(ii)] the functions $\lambda,\nu_{u}, u^{+}, u^{-}$ restricted to $S_{\eta}$ are continuous;
\item[(iii)] if we set 
\[
\Delta_{\eta}(r) := \sup_{x\in S_{\eta}} \frac{1}{r^{n}} \int_{B_{r}(x)} |u(y) - u_{x, \infty}(y)|\, dy\,,
\] 
we have $\displaystyle\lim_{r\to 0}\Delta_{\eta}(r) = 0$;
\item[(iv)] there exists $0<r_{\delta, \eta}<1$ such that 
\begin{equation} \label{eq:CdkAsymptotic}
|Du|(C_{\delta,k}(x,r)) \le c_n \delta\, |Du|(B_{r}(x)) 
\end{equation}
for all $0<r<r_{\delta, \eta}$, $0 < \delta < \frac{1}{3}$, $k > (\delta r)^{-1}$, and $x\in S_{\eta}$, where $C_{\delta, k}(x,r)$ is defined as in the previous step (here the dependence upon $x$ and $r$ is explicitly written for the sake of clarity) and $c_n = \max\{1, 7(n-1)\}$.
\end{itemize}
We remark that showing (i), (ii), and (iii) is standard. As for (iv), in the case $n \ge 2$ we notice that the asymptotic behaviour of $|Du|(C_{\delta,k})$ for $k > \frac{1}{r} \max \left \{ \frac{1}{\delta}, \frac{2}{1-\delta} \right \}$ and as $r\to 0$ is
\begin{align*}
|Du|(C_{\delta,k}) & = |D^d u|(C_{\delta,k}) + |D^j u|(C_{\delta,k}) \\
& =  \big (1 + o(1) \big ) \omega_{n-1} |u^{+}(x) - u^{-}(x)| \left[(r+2/k)^{n-1} - (r(1-\delta)-2/k)^{n-1}\right] \\
& =  \big (1 + o(1) \big ) \omega_{n-1} |u^{+}(x) - u^{-}(x)| r^{n-1}\left[(1+2/(kr))^{n-1} - ((1-\delta)-2/(kr))^{n-1}\right]\\
& =  \big (1 + o(1) \big ) \omega_{n-1} |u^{+}(x) - u^{-}(x)| r^{n-1}\delta \left[4/(k\delta r) +1 \right] (n-1) (1 + 2 \delta)^{n-2} \\ 
& \le 7 (n-1)\delta\, |Du|(B_{r}(x))\,,
\end{align*}
assuming $\delta$ sufficiently small so that $(1 + 2 \delta)^{n-2} \le \frac{7}{6}$. Then, it is easy to see that $\frac{1}{\delta} > \frac{2}{1-\delta}$ for all $\delta \in \left (0, \frac{1}{3} \right )$, so that we obtain (iv) for $n \ge 2$. In the case $n = 1$, we notice that by the finiteness of the total variation of $Du$, $S$ is necessarily a finite set, so that we can choose $S_\eta = S$ and easily check properties (i)--(iii) above. On the other hand, property (iv) follows from the observation that, for all $x\in S$,
\[
|Du|(C_{\delta,k}(x,r)) \le |Du|(B_{r}(x)\setminus \{x\}) \to 0\qquad \text{as }r\to 0\,,
\]
while $|Du|(B_{r}(x)) > \eps$ by the definition of $S$.

Fix now an open set $U$ containing $S_{\eta}$, then consider the family $\cF$ of balls centered in $S_{\eta}$ and contained in $U$, with radius $r$ so small that $\Delta_{\eta}(r) < \eps$. 
By Vitali-Besicovitch Covering Theorem, we can find a finite and mutually disjoint family of balls $\{B^{i}\}_{i=1}^{N}$, with $B^{i} = B_{r_{i}}(x_{i})\in \cF$, $0<r_{i}<r_{\delta, \eta}$, and $N$ depending on $\delta$ and $\eta$, such that
\[
\Haus{n-1}\left(S_{\eta}\cap \bigcup_{i=1}^{N}B^{i}\right) \ge (1-\eta)\Haus{n-1}(S_{\eta}) \ge (1-\eta)^{2}\Haus{n-1}(S)\,,
\]
where we have used property (i) in the last inequality. 
\medskip

\textit{Step three: total variation estimate.} 

We set $r_{0} = \min \{r_{i}:\ i=1,\dots,N\}$ and consider the sequence $u_{k}^{\lambda}$ constructed by replacing $u_{k}$ with $u^{\tau,B}_{k}$ inside each ball $B=B^{i}$, as obtained and described in the previous steps. Hereafter we show that the total variation of $Du^{\lambda}_{k}$ is controlled by that of $Du$, up to an error that goes to zero as $\delta\to 0$ (hence as $k\to\infty$). Indeed, if we set $C^{i}_{\delta} = B^{i} \setminus B_{(1-\delta)r_{i}}(x_{i})$, we can assume, up to a small perturbation of $r_{i}$, that $|Du|(\partial C^{i}_{\delta}) = 0$ for all $i$, hence thanks to the estimate \eqref{eq:mainlocalest} we obtain
\begin{align}
\nonumber
|Du^{\lambda}_{k}|(\Omega) &\le |Du_{k}|\left(\Omega\setminus \bigcup_{i=1}^{N} B^{i}\right) + \sum_{i=1}^{N}|Du^{\lambda}_{k}|(B^{i})\\\label{eq:stimaDulambda}
&\le |Du_{k}|\left(\Omega\setminus \bigcup_{i=1}^{N} B^{i}\right) 
+ \sum_{i=1}^{N} \left(|Du|(B^{i}_{\delta,k}) + \Big(C+\frac{2}{k\delta r_{i}}\Big)|Du|(C^{i}_{\delta,k})\right)\,.
\end{align}
Let now assume $k\ge (\delta r_{0})^{-1}$. Thanks to \eqref{eq:mainocaltest_1}, \eqref{eq:CdkAsymptotic} and recalling that the balls $B^{i}$ are mutually disjoint, we have
\begin{equation}\label{eq:sommacorone}
\sum_{i=1}^{N} \Big(C+\frac{2}{k\delta r_{i}}\Big)|Du|(C^{i}_{\delta,k}) \le (C+2)\sum_{i=1}^{N}|Du|(C^{i}_{\delta,k}) \le \widetilde{C}\delta \sum_{i=1}^{N} |Du|(B^{i}) \le \widetilde{C}\delta |Du|(\Omega)\,.
\end{equation}
Thanks to Theorem \ref{thm:Anzellotti_Giaquinta}, we have the strict convergence of $u_{k}$ to $u$ on $\Omega$, hence by the lower semicontinuity of $|Du_{k}|$ restricted to the open set $\bigcup_{i=1}^{N} B^{i}$ we obtain
\[
\limsup_{k\to\infty} |Du_{k}|\left(\Omega\setminus \bigcup_{i=1}^{N} B^{i}\right) = |Du|(\Omega) - \liminf_{k\to\infty} |Du_{k}|\left(\bigcup_{i=1}^{N} B^{i}\right) \le |Du|\left(\Omega\setminus \bigcup_{i=1}^{N} B^{i}\right)\,,
\]
hence by selecting $k$ large enough we can enforce 
\begin{equation}\label{eq:stimalimsup}
|Du_{k}|\left(\Omega\setminus \bigcup_{i=1}^{N} B^{i}\right) \le |Du|\left(\Omega\setminus \bigcup_{i=1}^{N} B^{i}\right) + \delta r_{0}\,.
\end{equation}
By combining \eqref{eq:stimaDulambda}, \eqref{eq:sommacorone}, and \eqref{eq:stimalimsup} we finally get
\begin{align}\nonumber
|Du^{\lambda}_{k}|(\Omega) &\le |Du|\left(\Omega\setminus \bigcup_{i=1}^{N} B^{i}\right) + \delta r_{0}
+ \sum_{i=1}^{N} |Du|(B^{i}) + \widetilde{C}\delta |D u|(\Omega)\\\label{eq:Dulambdafinal}
&= (1+\widetilde{C}\delta)|Du|(\Omega)\,.
\end{align}
This shows that the total variation of $Du^{\lambda}_{k}$ is arbitrarily close to the total variation of $Du$, up to choosing $k$ large enough. This will eventually lead to the $BV$-strict approximation property (see Step five). Arguing in a similar way, we can obtain an analogous upper bound for the area functional. This is due to the coincidence of the singular parts of the area and the total variation functionals (see \eqref{area_factor_decomposition_eq}) hence the previous construction leads to an estimate like \eqref{eq:Dulambdafinal} with $\sqrt{1 + |Du^{\lambda}_{k}|^{2}}$ and $\sqrt{1 + |Du|^{2}}$ replacing the total variation functionals.
%thanks to Lemma \ref{lem:est_area_functional} and \eqref{eq:BVstrict} we get
%\begin{align*}
%\sqrt{1 + |Du^{\lambda}_{k}|^2}(\Omega) & \le |\Omega| + (1+C	\delta) |D u|(\Omega) + 4\delta r_{0} \le (1+C\delta) \left ( |\Omega| + |D u|(\Omega) \right ) + 4\delta r_{0} \\
%& \le \sqrt{2} (1+C\delta) \sqrt{1 + |D u|^2}(\Omega) + 4\delta r_{0}.
%\end{align*}
\medskip

\textit{Step four: pointwise closeness of $u^{\lambda}_{k}$ to $u^{\lambda}$ on $S_{\eta}$.}

Let us fix $B^{i}=B_{r_{i}}(x_{i})$ for $i=1,\dots,N$, and choose $x\in B_{(1-\delta)r_{i}}(x_{i})\cap S_{\eta}$. We would like to prove that $u^{\lambda}_{k}(x)$ is close to $u^{\lambda}(x)$. We have
\begin{align*}
|u^{\lambda}_{k}(x) - u^{\lambda}(x)| &\le 
|u^{\lambda}_{k}(x) - u^{\lambda}(x_{i})| + |u^{\lambda}(x_{i}) - u^{\lambda}(x)| \le |u^{\lambda}_{k}(x) - u_{x,\infty}\ast \rho_{k}(x+v_{k}(x))| + \\
& + |u_{x,\infty}\ast \rho_{k}(x+v_{k}(x)) - u_{x_{i},\infty}\ast \rho_{k}(x_{i}+v_{k}(x_i))| + \\
&+ |u_{x_{i},\infty}\ast \rho_{k}(x_{i}+v_{k}(x_i)) - u^{\lambda}(x_{i})| + |u^{\lambda}(x_{i}) - u^{\lambda}(x)|\\
& =: A_{1}+A_{2}+A_{3}+A_{4}\,.
\end{align*}
First of all, by \eqref{eq:utaukappainx} and (iii) we obtain 
\begin{align*}
A_{1} &\le \int_{B_{1/k}(x+v_{k}(x))}\rho_{k}\big(x+v_{k}(x) - y\big)\, |u(y) - u_{x,\infty}(y)|\, dy\\
&\le k^{n}\|\rho\|_{L^{\infty}(B_1)} \int_{B_{2/k}(x)} |u(y) - u_{x,\infty}(y)|\,dy\ \to 0\qquad \text{as }k\to \infty\,.
\end{align*}
Then, $A_{3}=0$ by \eqref{eq:tauprop}, recalling that $v_{k}(y) = \frac{\tau(y)}{k}\nu_u(y)$, while thanks to (ii) the terms $A_{2}$ and $A_{4}$ are close to $0$ if $x$ is close to $x_{i}$, which in turn depends on the fact that $r_{i}$ is taken small enough. This shows that, by choosing $r_{i}$ small and $k\ge (\delta r_{0})^{-1}$ large enough, we can enforce the required pointwise closeness.
\medskip

\textit{Step five: conclusion.}
Let us fix three positive and infinitesimal sequences $\eps_{j}, \eta_{j}, \delta_{j}$, as well as a monotone sequence $A_{j}\Subset \Omega$ of Borel sets, such that $\bigcup_{j} A_{j} = \Omega$. Let $U_{j}$ be a sequence of open sets as in Step two, satisfying the extra condition $|U_{j}|< 2^{-j}$ for all $j$. For any integer $j\ge 1$ we can apply the previous steps with $A=A_{j}$ (Step one) and $U=U_{j}$ (Step two), and select from the initial sequence $(u_{k})_{k}$ a suitable subsequence, that we do not relabel, whose elements can be locally perturbed according to the procedure described in Step one. We shall perform this construction iteratively, so that the sequence that we extract at the $(j+1)$-th stage is also a subsequence of the one obtained at the $j$-th stage. By diagonal selection we obtain a sequence relabeled as $u^{\lambda}_{k}$. Owing to Step four, $u^{\lambda}_{k}$ converges to $u^{\lambda}$ in $\Haus{n-1}$-measure on the whole jump set $J_{u}$ as $k\to +\infty$, while by the choice of $U_{j}$ we have that $u^{\lambda}_{k}$ converges to $u$ in $L^{1}(\Omega)$. Hence, by Step three, $u^{\lambda}_{k}$ converges to $u$ in $BV(\Omega)$-strict, which is point (1) of the statement. Similarly, we deduce point (3). Up to a further extraction of a subsequence, we obtain the pointwise convergence $\Haus{n-1}$-a.e. on $\Omega$. Then, the bound $|u^\lambda_k(x)| \le \left( 1 + \frac{1}{k} \right ) \|u\|_{L^{\infty}(\Omega)}$ for all $x \in \Omega$ follows immediately from the definition of $u_k^\lambda$ in terms of $u^{\tau,B}_{k}$ and \eqref{def:u_tau_B}, thus concluding the proof of the theorem under the assumption $\|u\|_{L^\infty(\Omega)}<+\infty$. To obtain the complete proof, we apply the previous steps to the sequence of truncations $T_{m}(u)$, and obtain $(T_{m}(u)_{k}^{\lambda})_{k}$ for each $m\in \N$. The diagonal sequence $T_{k}(u)_{k}^{\lambda}$ can then be easily shown to satisfy all the required properties. Finally, point (5) can be proved as in Theorem \ref{thm:Anzellotti_Giaquinta}, given that the sequence $(u^k_\lambda)_k$ is obtained by a local perturbation of the approximating sequence $(u_k)_k$ from Anzellotti-Giaquinta approximation theorem.
\end{proof}

We state for later use a simple consequence of Theorem \ref{thm:smooth_lambda_approx}.

\begin{corollary} \label{cor:weak_area_conv}
Let $\lambda:\Omega\to [0,1]$ be a given Borel function. Let $u\in BV(\Omega)$ and $(u^\lambda_k)_{k \in \N}\subset C^\infty(\Omega)$ be the approximating sequence of Theorem \ref{thm:smooth_lambda_approx}. Then we have
\begin{equation*}
\sqrt{1 + |Du^{\lambda}_k|^2}\ \weakto\ \sqrt{1 + |Du|^2} \ \text{ in } \Mm(\Omega).
\end{equation*}
\end{corollary}

\begin{proof}
The lower semicontinuity on open sets is an immediate consequence of the fact that $u^\lambda_k \to u$ in $L^1(\Omega)$, thanks to point (1) of Theorem \ref{thm:smooth_lambda_approx}; while the upper semicontinuity on compact sets follows from the lower semicontinuity on the open sets and point (3) of Theorem \ref{thm:smooth_lambda_approx}.
\end{proof}

We conclude this section with an approximation result for the $\lambda$-pairings, which, as a byproduct, allows us to derive a different proof for the estimate \eqref{eq:pairing_estimate_lambda}, see Remark \ref{rem:alternate_proof_bound} below. In addition, we provide a useful alternative bound for the $\lambda$-pairing in our framework, a particular case of which was proved in \cite[Lemma 5.5]{scheven2016bv}.

\begin{theorem}\label{lemma:pairlambda-vs-area}
Let $F\in \DM^{\infty}(\Omega)$, $u\in BV(\Omega)$ and $\lambda:\Omega\to [0,1]$ be a Borel function. If $u^\lambda \in L^{1}(\Omega; |\div F|)$, then there exists a sequence $(u_j)_{j \in \N} \subset C^{\infty}(\Omega) \cap BV(\Omega) \cap L^\infty(\Omega)$ such that $u_j \to u$ in $BV(\Omega)$-strict and
\begin{equation} \label{eq:pairlambda_weak_conv}
(F \cdot \nabla u_j) \, \Leb{n} \weakto (F, Du)_{\lambda} \ \text{ in } \Mm(\Omega).
\end{equation}
If $u \in BV(\Omega) \cap L^\infty(\Omega)$, then $u F \in \DM^\infty(\Omega)$ and \eqref{eq:pairlambda_weak_conv} holds true for the approximating sequence $(u^\lambda_k)_{k \in \N}$ given by Theorem \ref{thm:smooth_lambda_approx}.

Finally, for all $a \ge \|F\|_{L^{\infty}(\Omega; \R^n)}$ we have
\begin{equation} \label{eq:pairlambda-vs-area}
|(F, Du)_{\lambda}| \le a\sqrt{1+|Du|^{2}} - \sqrt{a^2 -|F|^{2}}\, \Leb{n} \ \text{ on } \Omega.
\end{equation}
\end{theorem}

\begin{proof}
Thanks to Theorem \ref{thm:pairlambda_Leibniz}, we know that $\div(u F), (F, Du)_{\lambda} \in \Mm(\Omega)$ and that the Leibniz rule \eqref{eq:Leibniz_lambda_gen} holds true. In addition, Proposition \ref{prop:summability_lambda} implies that $\mul \in L^{1}(\Omega; |\div F|)$, given that $u^\lambda \in L^{1}(\Omega; |\div F|)$.

Now, we exploit Theorem \ref{thm:smooth_lambda_approx} in order to construct the smooth approximation $T_{N}(u)^{\lambda}_{k}$ of the truncation $T_{N}(u)$, for all $N > 0$.
Thanks to \eqref{eq:Leibniz_lambda_gen}, for all $k \in \N$ and $N > 0$, we have
\begin{equation*}
\div(T_{N}(u)^{\lambda}_{k} F) = T_{N}(u)^{\lambda}_{k} \div F + (F \cdot \nabla T_{N}(u)^{\lambda}_{k}) \Leb{n} \ \text{ on } \Omega.
\end{equation*}
Hence, Theorem \ref{thm:smooth_lambda_approx} implies that 
\begin{align*}
\left | \div(T_{N}(u)^{\lambda}_{k} F) \right |(\Omega) & \le \left ( 1 + \frac{1}{k} \right ) N |\div F|(\Omega) + \|F\|_{L^{\infty}(\Omega; \R^n)} \sup_{k \in \N} \|\nabla T_{N}(u)^{\lambda}_{k}\|_{L^1(\Omega; \R^n)}.
\end{align*}
Therefore, we conclude that $(\div(T_{N}(u)^{\lambda}_{k} F))_{k \in \N}$ is bounded sequence in $\Mm(\Omega)$ for each fixed $N > 0$, and, since $\div(T_{N}(u)^{\lambda}_{k} F)$ converges to $\div(T_N(u) F)$ in the sense of distributions, we deduce that it also weakly converges in the sense of Radon measures. In addition, Theorem \ref{thm:smooth_lambda_approx} and the Lebesgue theorem with respect to the measure $|\div F|$ imply that
\begin{equation*}
T_{N}(u)^{\lambda}_{k} \, \div F  \weakto T_{N}(u)^{\lambda} \div F  \ \text{ in } \Mm(\Omega).
\end{equation*}
All in all, we see that
\begin{align}
(F \cdot \nabla T_{N}(u)^{\lambda}_{k}) \Leb{n} & = \div(T_{N}(u)^{\lambda}_{k} F) - T_{N}(u)^{\lambda}_{k} \div F \nonumber \\
& \weakto \div(T_N(u) F) - T_{N}(u)^{\lambda} \div F = (F, DT_N(u))_\lambda \ \text{ as } \ k \to + \infty , \label{eq:pairing_T_N_u_k_lambda_conv_1}
\end{align}
given that \eqref{eq:Leibniz_lambda_gen} holds true for $T_N(u)$ and $F$, so that we get
\begin{equation} \label{eq:Leibniz_lambda_TN}
\int_{\Omega} \varphi \, d (F, D T_N(u))_{\lambda} = - \int_{\Omega} T_N(u) F \cdot \nabla \varphi \, dx - \int_{\Omega} \varphi (T_N(u))^\lambda \, d \div F
\end{equation}
for all $\varphi \in C^1_c(\Omega)$. Since the family of measures $\big ((F, D T_N(u))_{\lambda} \big)_{N>0}$ is uniformly bounded in $N > 0$ by \eqref{eq:pairing_estimate_lambda} and the fact that $|D T_N(u)| \le |Du|$, we see that, up to extracting a subsequence, we may pass to the limit as $N\to + \infty$. Thanks to Proposition \ref{prop:mul}, we see that the truncation $T_N(u)$ for $N > 0$ satisfies
\begin{equation*}
(T_N(u))^{\lambda}(x) \to u^{\lambda}(x) \text{ as } N \to + \infty \text{ and } |(T_N(u))^{\lambda}(x)| \le \mul(x) \text{ for } |\div F|\text{-a.e. } x \in \Omega,
\end{equation*}
so that the right hand side of \eqref{eq:Leibniz_lambda_TN} clearly converges to
\begin{equation*}
- \int_{\Omega} u F \cdot \nabla \varphi \, dx - \int_{\Omega} \varphi u^\lambda \, d \div F = \int_{\Omega} \varphi \, d (F, Du)_{\lambda},
\end{equation*}
thanks to Lebesgue's Dominated Convergence Theorem with respect to the measure $|\div F|$. Hence, we deduce that
\begin{equation} \label{eq:weak_conv_lambda_pairing_trunc}
(F, DT_N(u))_{\lambda}\weakto (F, Du)_{\lambda} \ \text{ in } \Mm(\Omega) \text{ as } N \to + \infty.
\end{equation}
All in all, we see that there exists a sequence $(k_j)_{j \in \N} \subset \N$ such that $u_j = T_{j}(u)^\lambda_{k_j} \in C^{\infty}(\Omega) \cap BV(\Omega) \cap L^\infty(\Omega)$ satisfies \eqref{eq:pairlambda_weak_conv}. If in addition $u \in BV(\Omega) \cap L^\infty(\Omega)$, then clearly $u F \in \DM^\infty(\Omega)$ and it is not necessary to consider the truncation of $u$, and therefore the first part of the argument above still holds true for the sequence $(u^\lambda_k)_{k \in \N}$ given by Theorem \ref{thm:smooth_lambda_approx}.

Finally, we deal with \eqref{eq:pairlambda-vs-area}. Due to the homogeneity of the $\lambda$-pairing in the first component, without loss of generality we assume $\|F\|_{L^{\infty}(\Omega; \R^n)} \le1$ and $a = 1$. For all $\phi\in C_{c}(\Omega)$ with $\phi \ge 0$, thanks to \eqref{eq:pairing_T_N_u_k_lambda_conv_1} and Corollary \ref{cor:weak_area_conv} applied to $T_N(u)$ for $N > 0$, we have
\begin{align*}
\int_{\Omega} \phi\, d(F, DT_N(u))_{\lambda} + \int_{\Omega} \phi\sqrt{1-|F|^{2}}\, dx & = \lim_{k \to + \infty} \int_{\Omega} \phi (F \cdot \nabla T_{N}(u)^{\lambda}_{k} + \sqrt{1-|F|^{2}})\, dx\\
& \le \liminf_{k \to + \infty} \int_{\Omega} \phi \left |(F, \sqrt{1-|F|^{2}}) \cdot (\nabla T_{N}(u)^{\lambda}_{k}, 1) \right |\, dx \\
& \le \lim_{k \to + \infty} \int_\Omega \phi \, \sqrt{1+|\nabla T_{N}(u)^{\lambda}_{k}|^{2}}\, dx \\
& = \int_{\Omega} \phi \, d \sqrt{1+|DT_N(u)|^{2}} \le \int_{\Omega} \phi \, d \sqrt{1+|Du|^{2}}.
\end{align*}
We see that \eqref{eq:weak_conv_lambda_pairing_trunc} yields
\begin{align*}
\int_{\Omega} \phi\, d(F, D u)_{\lambda} = \lim_{N \to + \infty} \int_{\Omega} \phi\, d(F, DT_N(u))_{\lambda} \le \int_{\Omega} \phi \, d \sqrt{1+|Du|^{2}} - \int_{\Omega} \phi\sqrt{1-|F|^{2}}\, dx.
\end{align*}
Due to the fact that $\phi \ge 0$, this implies that
\begin{equation*}
(F, Du)_{\lambda}^+ \le \sqrt{1+|Du|^{2}} - \sqrt{1 -|F|^{2}}\Leb{n}  \ \text{ on } \Omega.
\end{equation*}
By an analogous computation, in which the term $\displaystyle \int_\Omega \phi \, \sqrt{1-|F|^{2}} \, dx$ is subtracted instead of added, and $\phi$ is replaced with $-\phi$, we obtain
\[
-\int_{\Omega} \phi\, d(F, Du)_{\lambda} + \int_\Omega \phi\sqrt{1-|F|^{2}}\, dx \le \int_{\Omega} \phi \, d \sqrt{1+|Du|^{2}}\,,
\]
which implies
\[
(F, Du)_{\lambda}^- \le \sqrt{1+|Du|^{2}} - \sqrt{1-|F|^{2}}\Leb{n}  \ \text{ on } \Omega.
\]
All in all, we obtain \eqref{eq:pairlambda-vs-area}.
\end{proof}

\begin{remark} \label{eq:lambda_pair_est_scaling}
The optimal value of $a$ in \eqref{eq:pairlambda-vs-area} is not necessarily $a = \|F\|_{L^{\infty}(\Omega; \R^n)}$. Indeed, let us consider $F(x) = (c, 0, \dots, 0)$, for some $c > 0$, and $u(x) = x_1$. Then \eqref{eq:pairlambda-vs-area} reduces to
\begin{equation*}
c \le a \sqrt{2} - \sqrt{a^2 - c^2}
\end{equation*}
for all $a \ge c$. However, it is plain to see that the minimum of the function $g(a) = a \sqrt{2} - \sqrt{a^2 - c^2}$ on $[c, + \infty)$ is attained at $a = c \sqrt{2}$, and we have $g(c \sqrt{2}) = c$, while $g(c) = c \sqrt{2}$.
\end{remark}

\begin{remark} \label{rem:alternate_proof_bound}
We notice that we can also prove the estimate on the total variation of the $\lambda$-pairings \eqref{eq:pairing_estimate_lambda} as a consequence of the $\lambda$-approximation \eqref{eq:pairlambda_weak_conv}. Indeed, it is not difficult to check that $u_j \to u$ in $BV(\Omega)$-strict, thanks to Theorem \ref{thm:smooth_lambda_approx} and the well-known properties of the truncation operator. Given that the strict convergence implies the weak convergence $|\nabla u_j| \Leb{n} \weakto |Du|$, for all $\phi\in C_{c}(\Omega)$ we get
\begin{align*}
\int_{\Omega} \phi \, d (F, Du)_\lambda & = \lim_{j \to + \infty} \int_{\Omega} \phi \, (F \cdot \nabla u_j) \, dx \\
& \le \|F\|_{L^\infty(\Omega; \R^n)} \lim_{j \to + \infty} \int_{\Omega} |\phi| |\nabla u_j| \, dx = \|F\|_{L^\infty(\Omega; \R^n)} \int_\Omega |\phi| \, d |Du|.
\end{align*}
Then, this easily implies $|(F, Du)_\lambda | \le \|F\|_{L^\infty(\Omega; \R^n)} |D u|$ on $\Omega$.
\end{remark}

\section{Perimeter bounds and admissible measures}\label{sec:admissible_measures}

%From the weak formulation of the PMCM equation, see Definition \ref{def:PMCM}, we can immediately obtain a necessary condition for the existence of a solution.
In the following we shall use the notation $\MH(\Omega)$ introduced in \eqref{def:MH}.

We start by giving a simple characterization of measures which are the distributional divergence of an essentially bounded vector field.

\begin{lemma} \label{necessary_cond} 
Let $\mu \in \mathcal{M}(\Omega)$ and assume that there exists $F \in L^{\infty}(\Omega; \R^n)$ such that $\div F = \mu$. Then $\mu$ enjoys the following properties:
\begin{enumerate}
\item $\mu \in \MH(\Omega)$, 
\item $\max\{|\mu(E^{1})|, |\mu(E^{1} \cup \redb E)|\} \le \|F\|_{L^{\infty}(\Omega; \R^n)} \Per(E)$ for all sets $E \Subset \Omega$ of finite perimeter in $\R^n$,
\item if in addition $\Omega$ is weakly regular, then we have $$\max\{|\mu(E^{1} \cap \Omega)|, |\mu((E^{1} \cup \redb E) \cap \Omega)|\} \le \|F\|_{L^{\infty}(\Omega; \R^n)} \Per(E)$$ for all sets $E\subset \Omega$ of finite perimeter in $\R^n$.
\end{enumerate}
\end{lemma}
\begin{proof} 
It is clear that $F \in \DM^{\infty}(\Omega)$, and so $\mu = \div F \in \MH(\Omega)$ by \cite[Theorem 3.2]{Silhavy1}. Then, (2) and (3) are easy consequences of the Gauss--Green formulas, see Theorem \ref{thm:GG_app}. Indeed, by \eqref{eq:GG_1} and \eqref{eq:GG_2} we get
\begin{equation*} 
|\mu(E^1 \cap \Omega)| = |\div F(E^{1} \cap \Omega)| \le \|F\|_{L^{\infty}(E; \R^{n})} \Per(E) \le \|F\|_{L^{\infty}(\Omega; \R^n)} \Per(E), 
\end{equation*}
and
\begin{equation*}
|\mu((E^1 \cup \redb E) \cap \Omega)|  = |\div F((E^{1} \cup \redb E) \cap \Omega)| \le \|F\|_{L^{\infty}(\Omega; \R^n)} \Per(E).
\end{equation*}
\end{proof}

In analogy with the bounds obtained in Lemma \ref{necessary_cond} for measure which are the divergence of a $\DM^\infty$ vector field, we define a general class of measures satisfying similar bounds with respect to the perimeter.

\begin{definition}\label{def:muLper}
Given $\mu \in \Mm(\Omega)$ and $L > 0$, we say that $\mu$ belongs to $\PB_{L}(\Omega)$ if 
\begin{equation} \label{eq:subcritical_cond}
|\mu(E^{1} \cap \Omega)| \le L\, \Per(E)\, \quad \text{ for all measurable sets } \, E\subset \Omega.
\end{equation}
We also set $\PB(\Omega) := \bigcup_{L > 0} \PB_{L}(\Omega)$ and, if $\mu \in \PB(\Omega)$, we say that $\mu$ satisfies a {\em perimeter bound condition}.
\end{definition}

\begin{remark} \label{necessary_cond_ball} 
If $\mu\in \PB(\Omega)$, then there exists $L > 0$ such that $|\mu(B_r(x))| \le L n \omega_{n} r^{n - 1}$ for all $x \in \Omega$ and $r > 0$ small enough so that $B_r(x) \subset \Omega$. In particular, by \cite[Theorem 4.2 and Corollary 4.3]{Phuc_Torres} we have $\mu \in \MH(\Omega)$.
\end{remark}

It is interesting to notice that, under some mild regularity assumptions on $\Omega$, the perimeter bound condition \eqref{eq:subcritical_cond} needs only to be satisfied on open sets with smooth boundary compactly contained in $\Omega$.

\begin{lemma} \label{lem:open_sets_PB}
Let $\Omega$ be weakly regular, $\mu \in \Mm(\Omega)$ and $L > 0$. Then we have $\mu \in \PB_{L}(\Omega)$ if and only if
\begin{equation} \label{eq:subcritical_cond_smooth}
|\mu(A)| \le L\, \Per(A)\,\quad \text{ for all open sets with smooth boundary } \, A \Subset \Omega.
\end{equation}
\end{lemma}

\begin{proof}
Clearly, any $\mu \in \PB_{L}(\Omega)$ satisfies \eqref{eq:subcritical_cond_smooth}. As for the opposite implication, we start by noticing that, if $\mu$ satisfies \eqref{eq:subcritical_cond_smooth}, then $\mu \in \MH(\Omega)$, thanks to Remark \ref{necessary_cond_ball}. By \cite[Theorem 1.1]{MR3314116}, for all $\eps > 0$ there exists an increasing family of open sets with smooth boundary $(\Omega_\eps)_{\eps > 0}$ such that $\Omega_\eps \Subset \Omega$, $\bigcup_{\eps > 0} \Omega_\eps = \Omega$ and $\lim_{\eps \to 0^+} P(\Omega_\eps) = P(\Omega)$. 
%In particular, for all $\eta > 0$ there exists $\eps = \eps(\eta) > 0$ such that
%\begin{equation} \label{eq:mu_Omega_eta_eps}
%|\mu|(\Omega \setminus \Omega_\eps) < \eta.
%\end{equation}
We consider now $E \subset \Omega$ be a set of finite perimeter, and we set $E_\eps = E \cap \Omega_\eps$. Due to the smoothness of $\Omega_\eps$, we obtain that $E_\eps^1 = E^1 \cap \Omega_\eps$. Since $\mu \in \MH(\Omega)$, we can apply \cite[Theorem 3.1]{Comi_Torres} to conclude that there exists a smooth approximation $(E_{\eps, k})_{k \in \N}$ of the set of finite perimeter $E_\eps$ such that
\begin{equation*}
\lim_{k \to + \infty} \mu(E_{\eps, k}) = \mu(E_\eps^1) = \mu(E^1 \cap \Omega_\eps) \ \text{ and } \ \lim_{k \to + \infty} P(E_{\eps, k}) = P(E \cap \Omega_\eps).
\end{equation*}
In particular, the construction performed in \cite{Comi_Torres} implies that $E_{\eps, k} \Subset \Omega$, given that $E_\eps \Subset \Omega$. Hence, we can apply \eqref{eq:subcritical_cond_smooth} to $E_{\eps, k}$, and pass to the limit as $k \to + \infty$ in order to get
\begin{equation*}
|\mu(E^1 \cap \Omega_\eps)| = \lim_{k \to + \infty} \mu(E_{\eps, k}) \le L \lim_{k \to + \infty} P(E_{\eps, k}) = L P(E \cap \Omega_\eps).
\end{equation*}
Now, we pass to the limit as $\eps \to 0^+$ and we employ the submodularity of the perimeter \cite[Proposition 3.38]{AFP} to obtain
\begin{align*}
|\mu(E^1 \cap \Omega)| & = \lim_{\eps \to 0^+} |\mu(E^1 \cap \Omega_\eps)| \le L \limsup_{\eps \to 0^+} \left (P(E) + P(\Omega_\eps) - P(E \cup \Omega_\eps) \right ) \\
& = L \left ( P(E) + P(\Omega) - \liminf_{\eps \to 0^+} P(E \cup \Omega_\eps) \right ) \le L \left ( P(E) + P(\Omega) - P(\Omega) \right ) = L P(E).
\end{align*}
Thus, we prove \eqref{eq:subcritical_cond}.
\end{proof}

We state now a basic result concerning nonnegative measures in the dual of $BV$.

\begin{lemma}\label{lem:pre_adm}
Let $\nu \in BV(\Omega)^{*}$ be nonnegative, and let $C > 0$ be such that
\begin{equation*}
\left | \int_\Omega u^{*} \, d \nu \right | \le C\, \|u\|_{BV(\Omega)}  \quad \text{ for all } u\in BV(\Omega) \cap L^\infty(\Omega).
\end{equation*}
Then for any $u\in BV(\Omega)$ we have $u^{\pm}\in L^{1}(\Omega;\nu)$, as well as $u^*, u^\lambda, \mul \in L^{1}(\Omega; \nu)$ for any Borel function $\lambda : \Omega \to [0, 1]$. In particular, we obtain
\begin{equation} \label{eq:ext_BV_dual}
\int_\Omega |u^{\lambda}| \, d \nu \le C\, \|u\|_{BV(\Omega)}  \quad \text{ for all } u\in BV(\Omega).
\end{equation}
\end{lemma}

\begin{proof}
By Theorem \ref{thm:smooth_lambda_approx} there exist two sequences $(u_{k}^{0})$ and $(u_{k}^{1})$ (corresponding to the choices $\lambda \equiv 0$ and $\lambda \equiv 1$) of smooth functions in $BV(\Omega)$ converging to $u$ in $BV(\Omega)$-strict, and respectively to $u^{-}$ and $u^{+}$ $\nu$-almost everywhere on $\Omega$. With a little abuse of notation, we shall write $u_{k}^{-} = u_{k}^{0}$ and $u_{k}^{+} = u_{k}^{1}$, from this point onwards. For any fixed $N\in \N$, the truncation operator $T_N$ is Lipschitz, so that the sequences of truncations $\big ( T_{N}(|u^{\pm}_{k}|) \big )$ converge to $T_{N}(|u^{\pm}|)$ $\nu$-almost everywhere on $\Omega$, as $k\to\infty$. By Lebesgue's Dominated Convergence Theorem we obtain
\begin{align*}
\int_{\Omega}T_{N}(|u^{\pm}|)\, d\nu &= \lim_{k\to\infty} \int_{\Omega} T_{N}(|u^{\pm}_{k}|)\, d\nu \le C\limsup_{k\to\infty} \|T_{N}(|u_{k}|)\|_{BV(\Omega)} \\
&\le C\limsup_{k\to\infty} \|\,|u_{k}|\,\|_{BV(\Omega)} \le C\limsup_{k\to\infty} \|u_{k}\|_{BV(\Omega)} = C \|u\|_{BV(\Omega)}\,.
\end{align*}
By monotone convergence we can take the limit as $N\to\infty$ and conclude
\begin{equation} \label{eq:u_pm_nu_summ}
\int_{\Omega}|u^{\pm}|\, d\nu \le C \|u\|_{BV(\Omega)}\,.
\end{equation}
This easily implies that $u^*, u^\lambda, \mul \in L^1(\Omega; \nu)$ for any Borel function $\lambda : \Omega \to [0, 1]$, and that \eqref{eq:ext_BV_dual} holds true.
\end{proof}

We introduce a special class of measures related to the dual of $BV$, which we call admissible in view of their role in the subsequent paper \cite{LeoComi}.

\begin{definition}\label{def:muadmissible}
We say that $\mu\in \MH(\Omega)$ is \emph{admissible} if $|\mu| \in BV(\Omega)^{*}$.
\end{definition}

\begin{lemma}\label{lem:2a2b-bis}
If $\mu \in \MH(\Omega)$ is admissible, then, for all $u\in BV(\Omega)$ and all Borel functions $\lambda : \Omega \to [0, 1]$, we have $u^\lambda, \mul \in L^{1}(\Omega; |\mu|)$. In particular, there exists a constant $C>0$ such that 
\begin{equation} \label{eq:admissibility_def}
\int_\Omega |u^{\lambda}| \, d|\mu| \le C\, \|u\|_{BV(\Omega)}  \quad \text{ for all } u\in BV(\Omega) \text{ and all Borel functions } \lambda : \Omega \to [0, 1].
\end{equation}
\end{lemma}

\begin{proof}
Since $|\mu| \in BV(\Omega)^*$, then Lemma \ref{lem:pre_adm} immediately yields the conclusion.
\end{proof}

\begin{remark} \label{rem:div_admissible_Leibniz}
We notice that, thanks to Lemma \ref{lem:2a2b-bis}, if the measure $\div F$ is admissible in Theorems \ref{thm:pairlambda_Leibniz}, \ref{thm:GG_boundary_domain_u} and \ref{lemma:pairlambda-vs-area}, then we do not need to require $u^\lambda \in L^{1}(\Omega; |\div F|)$. In other words, if $\div F$ is admissible, then the Leibniz rule and the integration by parts formula hold true for all $BV$ functions.
\end{remark}

\begin{remark} \label{rem:extremal_vs_admissible}
We notice that, if $\mu \in \MH(\Omega)$ is admissible, then we have also $\mu \in BV(\Omega)^{*}$. Indeed, due to Lemma \ref{lem:2a2b-bis}, we know that $u^* \in L^1(\Omega, |\mu|)$ for all $u \in BV(\Omega)$, and, by exploiting \eqref{eq:admissibility_def} for $\lambda = \frac{1}{2}$, we see that
\begin{equation*}
\left | \int_{\Omega} u^{*} \, d \mu \right | \le \int_{\Omega} |u^{*}| \, d |\mu| \le C \big\| u \big\|_{BV(\Omega)}.
\end{equation*}
We underline that, in general, the converse is not true, as shown in Remark \ref{rem:counterexample}. 
\end{remark}

In the following proposition, we explore the relations between the divergence-measure fields, the perimeter bound, and the admissibility condition.

\begin{proposition} \label{prop:adm_PB_div}
The following hold true:
\begin{enumerate}
\item If $\Omega$ is weakly regular and $F \in \DM^{\infty}(\Omega)$, then $\div F \in \PB_L(\Omega)$ for $L = \|F\|_{L^\infty(\Omega; \R^n)}$; if in addition $\Omega$ has Lipschitz boundary, then $\div F \in BV(\Omega)^*$.

\item If $\Omega$ is a bounded open set with Lipschitz boundary and $\mu \in \PB(\Omega)$, then there exists $F \in \DM^{\infty}(\Omega)$ such that $\div F = \mu$ on $\Omega$, there exists $\widetilde{L} > 0$ such that
\begin{equation} \label{eq:L_non_extreme}
|\mu(U \cap \Omega)| \le \widetilde{L} \, \Per(U) \quad \text{ for all bounded open sets } U\subset \R^n \text{ with smooth boundary,}
\end{equation} 
we have $\mu \in BV(\Omega)^*$, with its action given for all $u \in BV(\Omega) \cap L^\infty(\Omega)$ by
\begin{equation} \label{eq:action_functional_pairing}
\T_{\mu} (u) = \int_\Omega u^* \, d \mu = - (F, Du)_*(\Omega) - \int_{\partial \Omega} {\rm Tr}_{\partial \Omega}(u) {\rm Tr}^i(F, \partial \Omega) \, d \Haus{n-1}\,,
\end{equation}
and, if $\mu$ is admissible, formula \eqref{eq:action_functional_pairing} holds for all $u \in BV(\Omega)$.
 If $n = 1$, \eqref{eq:L_non_extreme} holds for all finite unions of open intervals.

\item If $n \ge 2$, $|\Omega| < + \infty$ and $\mu \in \MH(\Omega)$ is an admissible measure, then $\mu \in \PB(\Omega)$.

\item If $\mu \in \MH(\Omega)$ is an admissible measure, then there exists $F, \overline{F} \in \DM^{\infty}(\Omega)$ such that $\mu = \div F$ and $|\mu| = \div \overline{F}$ on $\Omega$.

\item If $n = 1$ and $\mu \in \Mm(\Omega)$, then $\mu$ is admissible and there exists $f, \overline{f} \in BV(\Omega)$ such that $\mu = D f$ and $|\mu| = D \overline{f}$ on $\Omega$; if in addition $|\R \setminus \Omega| > 0$, then $\Mm(\Omega) = \PB(\Omega)$.
\end{enumerate}
\end{proposition}

\begin{proof}
The first part of point (1) is an immediate consequence of Lemma \ref{necessary_cond}. The second is a consequence of the integration by parts formula \eqref{eq:GG_boundary_domain_u} for $\lambda \equiv \frac{1}{2}$: for all $u \in BV(\Omega) \cap L^\infty(\Omega)$  we exploit the continuity of the trace operator, \eqref{eq:pairing_estimate_lambda} and \eqref{eq:GG_boundary_domain} to get
\begin{align*}
\left | \int_{\Omega} u^* \, d \div F \right | & \le |(F, Du)_*(\Omega)| + \left | \int_{\partial \Omega} {\rm Tr}_{\partial \Omega}(u) {\rm Tr}^i(F, \partial \Omega) \, d \Haus{n-1} \right | \\
& \le \|F\|_{L^\infty(\Omega; \R^n)} \left ( |Du|(\Omega) + C \|u\|_{BV(\Omega)} \right ).
\end{align*}
 
As for point (2), the existence of $F \in \DM^{\infty}(\Omega)$ such that $\div F = \mu$ on $\Omega$ follows from \cite[Theorem 7.4]{Phuc_Torres}.
Now, if $n = 1$, we notice that
\begin{equation*}
\left | \int_\Omega u^\lambda \, d \mu \right | \le |\mu|(\Omega) \|u\|_{L^\infty(\Omega)} \le C \|u\|_{BV(\Omega)} \quad \text{ for all } \ u \in BV(\Omega) \text{ and } \lambda : \Omega \to [0, 1] \text{ Borel},
\end{equation*}
thanks to the embedding $BV(\Omega) \subset L^\infty(\Omega)$. Therefore, if $\lambda \equiv \frac{1}{2}$, we see that $\mu \in BV(\Omega)^*$. In addition, if $U \subset \R$ is a finite union of open intervals, we obtain
\begin{equation*}
\left | \mu(U \cap \Omega) \right | \le |\mu|(\Omega) \le |\mu|(\Omega) P(U),
\end{equation*}
since $P(U) \ge 1$.
If instead $n \ge 2$, we consider an open bounded set $U$ with smooth boundary and apply \eqref{eq:GG_boundary_domain_u} to the vector field $F$ and to the function $\chi_U \in BV(\Omega) \cap L^\infty(\Omega)$, by choosing $\lambda \equiv 0$. Given that $U = U^1$ due to the smoothness of the boundary, by \eqref{eq:repr_chi_E} we get $\chi_U^\lambda = \chi_U^- = \chi_U$ and
\begin{equation*}
\int_\Omega \chi_U \, d \div F = - (F, D\chi_U)_0(\Omega) - \int_{\partial \Omega} {\rm Tr}_{\partial \Omega}(\chi_U) {\rm Tr}^i(F, \partial \Omega) \, d \Haus{n-1}.
\end{equation*}
Since $\div F = \mu$, we exploit the continuity of the trace operator, \eqref{eq:pairing_estimate_lambda} and \eqref{eq:isoper_omega} to obtain
\begin{align*}
|\mu(U \cap \Omega)| & \le \|F\|_{L^\infty(\Omega; \R^n)} \left ( |D\chi_U|(\Omega) + \|{\rm Tr}_{\partial \Omega}(\chi_U)\|_{L^1(\partial \Omega; \Haus{n-1})} \right ) \\
& \le \|F\|_{L^\infty(\Omega; \R^n)} \left ( \Per(U, \Omega) + C \|\chi_U\|_{BV(\Omega)} \right ) \\
& \le \|F\|_{L^\infty(\Omega; \R^n)} \left ( \Per(U, \Omega) + C (c_n |\Omega|^{\frac{1}{n}} + 1) \Per(U) \right ).
\end{align*}
All in all, we obtain \eqref{eq:L_non_extreme} with $\widetilde{L} = \|F\|_{L^\infty(\Omega; \R^n)} ( 1 + C (c_n |\Omega|^{\frac{1}{n}} + 1) )$. Thus, we exploit \cite[Theorem 8.2]{Phuc_Torres}\footnote{In this result, the authors extend $\mu$ to $0$ on $\R^n \setminus \Omega$, which is equivalent to taking the intersection with $\Omega$ on the left hand side of \eqref{eq:L_non_extreme}. Also, by closely inspecting the proof, we notice that it is sufficient to take bounded open sets with smooth boundary.} to conclude that $\mu \in BV(\Omega)^*$. Finally, \eqref{eq:action_functional_pairing} is an immediate consequence of Theorem \ref{thm:GG_boundary_domain_u}, combined with Lemma \ref{lem:2a2b-bis} and Remark \ref{rem:div_admissible_Leibniz} if $\mu$ is admissible.

Then, we notice that, if $\mu \in \MH(\Omega)$, by \eqref{eq:repr_chi_E} we have
\begin{equation} \label{eq:chi_E_star_mu}
\left | \int_\Omega \chi_E^* \, d |\mu| \right | = |\mu|(E^1 \cap \Omega) + \frac{1}{2} |\mu|(\partial^* E \cap \Omega) \ge |\mu(E^1 \cup \Omega)|
\end{equation}
for any measurable set $E \subset \Omega$ of finite perimeter in $\R^n$.

Let now $n \ge 2$, $|\Omega| < + \infty$, $\mu \in \MH(\Omega)$ be an admissible measure and $C > 0$ be such that \eqref{eq:admissibility_def} holds true. Hence, we choose $u = \chi_E$ for some measurable set $E \subset \Omega$ of finite perimeter in $\R^n$ and we exploit \eqref{eq:isoper_omega} to obtain
\begin{equation*}
\left | \int_\Omega \chi_E^* \, d |\mu| \right | \le C \| \chi_E\|_{BV(\Omega)} \le C (c_n |\Omega|^{\frac{1}{n}} + 1) P(E).
\end{equation*}
Thus, in the light of \eqref{eq:chi_E_star_mu}, we get $\mu \in \PB_{L}(\Omega)$ for $L = C (c_n |\Omega|^{\frac{1}{n}} + 1)$, and this proves point (3).

Point (4) is stated in \cite[Lemma 7.3]{Phuc_Torres} in the case $\Omega$ is an open bounded set with Lipschitz boundary; however, given that the proof is completely analogous to the one in the case $\Omega = \R^n$ (see \cite[Lemma 4.1]{Phuc_Torres}), we stress the fact that this representation result holds true without any additional assumption on the set $\Omega$.

Finally, let $n = 1$. We notice that $\MH(\Omega) = \Mm(\Omega)$, and, arguing as above, we see that
\begin{equation*}
\left | \int_\Omega u^* \, d |\mu| \right | \le |\mu|(\Omega) \|u\|_{L^\infty(\Omega)} \le C \|u\|_{BV(\Omega)} \quad \text{ for all } \ u \in BV(\Omega),
\end{equation*}
thanks to the embedding $BV(\Omega) \subset L^\infty(\Omega)$, which shows that $|\mu| \in BV(\Omega)^*$. Hence, we exploit point (4) and the fact that $\DM^\infty(\Omega) = BV(\Omega)$. Then, if $|\R \setminus \Omega| > 0$, we have $P(E) \ge 1$ as long as $|E| > 0$, so that we get
\begin{equation*}
\left | \int_\Omega \chi_E^* \, d |\mu| \right |  \le |\mu|(\Omega) P(E).
\end{equation*}
Combining this bound with \eqref{eq:chi_E_star_mu} we obtain $\mu \in \PB_{L}(\Omega)$ for $L = |\mu|(\Omega)$, and so we conclude that $\Mm(\Omega) = \PB(\Omega)$.
\end{proof}

As an immediate consequence of Proposition \ref{prop:adm_PB_div}(2), we obtain the following refinement of \cite[Theorem 8.2]{Phuc_Torres}.

\begin{corollary} \label{cor:PT}
Let $\Omega$ be a bounded open set with Lipschitz boundary and $\mu \in \Mm(\Omega)$. Then we have $\mu \in BV(\Omega)^*$ if and only if there exists $L > 0$ such that
\begin{equation} \label{eq:PT}
|\mu(U)| \le L\, \Per(U)\,\quad \text{ for all open sets with smooth boundary } \, U \Subset \Omega.
\end{equation}
\end{corollary}

\begin{proof}
If $\mu \in BV(\Omega)^*$, then \cite[Theorem 8.2]{Phuc_Torres} implies \eqref{eq:PT}. On the other hand, combining Lemma \ref{lem:open_sets_PB} and \eqref{eq:PT}, we get $\mu \in \PB_L(\Omega)$. Therefore, by Proposition \ref{prop:adm_PB_div}(2) we conclude that $\mu \in BV(\Omega)^*$.
\end{proof}

In the light of Proposition \ref{prop:adm_PB_div}(5), we see that in the one dimensional case every Radon measure is admissible. This fails to be true if $n \ge 2$, even assuming the perimeter bound condition, as we show in the following remark.

\begin{remark} \label{rem:counterexample}
There exist measures in $\PB(\Omega)$ which are not admissible. To see this we can closely follow \cite[Proposition 5.1]{Phuc_Torres}, which shows the existence of a measure $\mu \in BV(\R^n)^*$ such that $|\mu| \notin BV(\R^n)^*$. We let $\mu_\eps = f_\eps \, \Leb{n}$, where
\begin{equation*}
f_\eps(x) = \eps |x|^{-1 - \eps} \sin{(|x|^{-\eps})} + (n-1) |x|^{-1} \cos{(|x|^{-\eps})} \ \text{ for } x \neq 0
\end{equation*} 
and some $\eps \in (0, n-1)$. Then, we have $f_\eps = \div F_\eps$ for 
\begin{equation*}
F_\eps(x) = \frac{x}{|x|} \cos{(|x|^{-\eps})}.
\end{equation*}
Clearly, $F_\eps \in \DM^\infty_{\rm loc}(\R^n)$ and $\|F_\eps\|_{L^\infty(\R^n; \R^n)} = 1$, and so, due to Proposition \ref{prop:adm_PB_div}(1), we have $\mu_\eps \in BV(\Omega)^* \cap \PB_1(\Omega)$ for any open bounded set $\Omega$ with Lipschitz boundary.

Then, we consider the function $u_\eps(x) = |x|^{-n + 1 + \eps}$, which satisfies $u_\eps \in W^{1,1}_{\rm loc}(\R^n)$. We notice that $u_\eps \notin L^1(B_r; |\mu_\eps|)$ for any $r > 0$. Indeed, while it is easy to check that
\begin{equation*}
\int_{B_r} \frac{|\cos{(|x|^{-\eps})}|}{|x|^{n - \eps}} \, dx < + \infty,
\end{equation*}
we see that, if we set $c_n = \Haus{n-1}(\partial B_1)$,
\begin{equation*}
\int_{B_r} \frac{|\sin{(|x|^{-\eps})}|}{|x|^{n}} \, dx = c_n \int_0^r \frac{|\sin{(\rho^{-\eps})}|}{\rho} \, d \rho = [\rho^{-\eps} = t] = c_n  \int_{r^{-\frac{1}{\eps}}}^{+\infty} \frac{1}{\eps} \frac{|\sin{(t)}|}{t} \, d t = + \infty,
\end{equation*}
due to the well-known fact that the function $t \to \frac{\sin{(t)}}{t} \notin L^1((\delta, + \infty))$ for all $\delta > 0$. This proves that $u_\eps f_\eps \notin L^1(B_r)$, and therefore $u_\eps \notin L^1(B_r; |\mu_\eps|)$. More in general, for any open set $\Omega$ containing the origin we have $u_\eps \notin L^1(\Omega; |\mu_\eps|)$, and therefore, due to Lemma \ref{lem:2a2b-bis}, we obtain $|\mu_\eps| \notin BV(\Omega)^*$.
\end{remark}

The previous remark shows that we cannot drop the admissibility condition on $\mu$, if we want to ensure that $u^\lambda \in L^1(\Omega; |\mu|)$ for all $u \in BV(\Omega)$ and $\lambda : \Omega \to [0, 1]$ Borel, even in the case $\Omega$ is a bounded open set with Lipschitz boundary. However, even without the admissibility condition, it is enough to assume that at least one $\lambda$-representative of $u \in BV(\Omega)$ is $|\mu|$-summable, in order to ensure this summability for all of them.

\begin{remark} \label{rem:summability_lambda_equiv}
If $\mu \in \PB(\Omega)$, Proposition \ref{prop:adm_PB_div}(2) implies that $\mu = \div F$ for some $F \in \DM^{\infty}(\Omega)$, hence we can apply Proposition \ref{prop:summability_lambda} to conclude that, given some Borel functions $\lambda_1, \lambda_2 : \Omega \to [0, 1]$, we have $u^{\lambda_1} \in L^1(\Omega; |\mu|)$ if and only if $u^{\lambda_2} \in L^1(\Omega; |\mu|)$. In addition, for any Borel function $\lambda : \Omega \to [0, 1]$, we have $u^\lambda \in L^1(\Omega; |\mu|)$ if and only if $\mul \in L^1(\Omega; |\mu|)$, again by Proposition \ref{prop:summability_lambda}.
\end{remark}

We provide a few basic examples of admissible measures, which, in the light of Proposition \ref{prop:adm_PB_div}, also belong to in $\PB(\Omega)$, as long as $|\Omega| < + \infty$.

\begin{example}\label{ex:tMomega} \rm
Let $\Omega$ be a bounded open set with Lipschitz boundary. 
We consider $\mu\in \Mm(\Omega)$ defined as
\[
\mu = h\, \Leb{n} + \gamma\, \Haus{n-1}\restrict \Gamma
\] 
where $h\in L^{q}(\Omega)$ for some $q>n$, $\gamma \in L^{\infty}(\Omega; \Haus{n-1}\restrict \Gamma)$, 
and $\Gamma\subset \Omega$ is a compact set such that there exist $\Lambda, \overline{\rho}>0$ for which
\begin{equation}\label{eq:mudensityestimate}
\Haus{n-1}(\Gamma \cap B_{\rho}(x)) \le \Lambda \rho^{n-1}\qquad \forall\, 0<\rho<\overline{\rho} \text{ and }\forall\, x\in \Gamma\,.
\end{equation}
We claim that $\mu$ is admissible, and we prove it by showing that both the measures $h\, \Leb{n}$ and $\gamma\, \Haus{n-1}\restrict \Gamma$ are admissible. In particular, this implies that all these measures belong to $\PB(\Omega)$, thanks to Proposition \ref{prop:adm_PB_div}. We first point out that, thanks to the Sobolev embedding of $BV(\Omega)$ into $L^{\frac{n}{n-1}}(\Omega)$, we get
\begin{equation*}
\left | \int_{\Omega} u \, |h| \, dx \right | \le \|h\|_{L^{n}(\Omega)} \|u\|_{L^{\frac{n}{n-1}}(\Omega)} \le C  \|h\|_{L^{n}(\Omega)} \|u\|_{BV(\Omega)}  \  \text{ for all } u \in BV(\Omega),
\end{equation*}
This shows that $h \Leb{n}$ is an admissible measure.

Then, we notice that the reason for requiring \eqref{eq:mudensityestimate} is that it ensures a crucial continuity property of the upper and lower trace operators from $BV$ to $L^{1}(\Gamma; \Haus{n-1})$. Indeed, there exists a constant $C>0$ depending only on $\Omega, \Haus{n-1}(\Gamma), \overline{\rho}$ and $\Lambda$ such that 
\begin{equation} \label{eq:Gammadual}
\int_{\Gamma} |u^{\pm}|\, d\Haus{n-1} \le C \|u\|_{BV(\Omega)} \text{ for all } u \in BV(\Omega).
\end{equation}

In order to show \eqref{eq:Gammadual}, we consider the zero-extension of $u$ to $\R^n$, which we denote by $u_0$: we have $u_0 \in BV(\R^n)$ and
\begin{equation} \label{eq:ext_BV}
\|u_{0}\|_{BV(\R^{n})}\le c\|u\|_{BV(\Omega)},
\end{equation}
for some $c = c(\Omega) > 0$ (see for instance \cite[Lemma 5.10.4]{Ziemer}).
Due to \eqref{eq:mudensityestimate}, we can apply \cite[Theorem 3.86]{AFP} to get
\begin{equation}\label{eq:traceineq}
\int_{\Gamma} |u^{\pm}|\, d\Haus{n-1} \le C \|u_{0}\|_{BV(\R^{n})} 
\end{equation}
where $C>0$ is a constant depending only on $\Haus{n-1}(\Gamma), \overline{\rho}$ and $\Lambda$. Combining \eqref{eq:ext_BV} and \eqref{eq:traceineq}, we obtain \eqref{eq:Gammadual}, which in particular implies $\Haus{n-1} \restrict \Gamma \in BV(\Omega)^*$ and that $u^{\pm} \in L^1(\Gamma; \Haus{n-1})$. Hence, for all $u \in BV(\Omega)$ we obtain
\begin{equation*}
\left | \int_{\Gamma} u^{*} \, |\gamma| \, d \Haus{n-1} \right | \le \int_{\Gamma} |u^*| \, |\gamma| \, d \Haus{n-1} \le C \|\gamma\|_{L^{\infty}(\Gamma; \Haus{n-1})} \|u\|_{BV(\Omega)},
\end{equation*}
where $C>0$ is a constant depending only on $\Omega, \Haus{n-1}(\Gamma), \overline{\rho}$ and $\Lambda$. Therefore, the measure $\gamma \, \Haus{n-1} \res \Gamma$ is admissible. All in all, this proves that $\mu$ is an admissible measure.
\end{example}

\begin{example} \rm \label{ex:Virginia}
Let $n \ge 2$ and $\mu = h \Leb{n}$, where $$h(x) = \frac{(n-1)}{|x-x_0|} \ \text{ for } x \neq x_0,$$ for some $x_0 \in \R^n$. We notice that $h \in L^q_{\rm loc}(\R^n)$ only for $q \in [1, n)$. However, it is easy to check that $\mu = \div F$, where 
$$F(x) = \frac{(x-x_0)}{|x-x_0|} \ \text{ for } x \neq x_0.$$ 
By Proposition \ref{prop:adm_PB_div}(1), for any open bounded set $\Omega$ with Lipschitz boundary we have $\mu \in BV(\Omega)^*$, so that $\mu$ is indeed an admissible measure, given that $\mu \ge 0$. In addition, since $\|F\|_{L^\infty(\R^n; \R^n)} = 1$, we see that $\mu \in \PB_1(\Omega)$, by Lemma \ref{necessary_cond}.
\end{example}

Given that the admissibility of a measure depends on its total variation, it seems natural to characterize it in terms of properties of the positive and negative parts of the given measure. In the following lemma we prove that any measure whose positive and negative parts belong to $\PB(\Omega)$ is indeed admissible.

\begin{lemma} \label{lem:Schmidt_non_extremality}
Let $\Omega$ be an open bounded set with Lipschitz boundary and let $\mu \in \Mm(\Omega)$. Assume that its positive and negative parts, $\mu^+$ and $\mu^-$, belong to $\PB_{L_+}(\Omega)$ and $\PB_{L_{-}}(\Omega)$ for some $L_{+},L_{-} > 0$, respectively. Then $\mu$ is admissible and belongs to $\PB_{L}(\Omega)$, where $L = \max\{L_+, L_{-}\}$. In particular, if $\mu \ge 0$ and $\mu \in \PB(\Omega)$, then it is admissible.
\end{lemma}

\begin{proof}
Thanks to Proposition \ref{prop:adm_PB_div}(2), we know that $\mu^+, \mu^- \in BV(\Omega)^*$. Hence, a simple application of the triangle inequality implies that $|\mu| \in BV(\Omega)^*$. Let now $L_+, L_- > 0$ be the constants for which $\mu^+$ and $\mu^-$ satisfy \eqref{eq:subcritical_cond}, respectively. We notice that, for all measurable sets $E \subset \Omega$, we have
\begin{equation*}
|\mu(E^1 \cap \Omega)| = |\mu^+(E^1 \cap \Omega) - \mu^-(E^1 \cap \Omega)| \le \max \{\mu^+(E^1 \cap \Omega), \mu^-(E^1 \cap \Omega) \} \le \max\{L_+, L_{-}\} P(E),
\end{equation*}
so that $\mu \in \PB_{L}(\Omega)$, for $L = \max\{L_+, L_{-}\}$.
\end{proof}

%\begin{remark} \label{rem:mu_div_T_admissible}
%If $\mu \in \MH(\Omega)$ is an admissible measure, then there exists $F \in \DM^{\infty}(\Omega)$ such that $\mu = \div F$ on $\Omega$. This result is stated in \cite[Lemma 7.3]{Phuc_Torres} in the case $\Omega$ is an open bounded set with Lipschitz boundary; however, given that the proof is completely analogous to the one in the case $\Omega = \R^n$ (see \cite[Lemma 4.1]{Phuc_Torres}), we stress the fact that this representation result holds true without any additional assumption on the set $\Omega$. Furthermore, there exists $\widetilde{F} \in \DM^{\infty}(\Omega)$ such that $|\mu| = \div \widetilde{F}$ on $\Omega$.
%\end{remark}

A relevant consequence of the results explored so far is that, for admissible measures $\mu$, the functional $\T_{\mu, \lambda}$ defined in \eqref{def:T_mu_lambda} can be naturally extended to all functions in $BV(\Omega)$.

\begin{lemma} \label{lem:lambda_int_approx}
Let $\mu \in \MH(\Omega)$ be admissible and $\lambda: \Omega \to [0, 1]$ be a Borel function. Then the functional $\T_{\mu, \lambda} : BV(\Omega) \to \R$ given by \eqref{def:T_mu_lambda}
%$$\T_{\mu, \lambda}(u) = \int_{\Omega} u^\lambda \, d \mu$$ 
is well-defined.
Then, given $u \in BV(\Omega)$, if we consider the truncation of $u$, $T_N(u)$ for $N > 0$, and its smooth approximation $T_{N}(u)^{\lambda}_{k}$ given by Theorem \ref{thm:smooth_lambda_approx}, we have
\begin{equation} \label{eq:conv_trunc_lambda}
\lim_{N \to + \infty} \int_{\Omega} T_N(u)^{\lambda} \, d\mu = \int_{\Omega} u^{\lambda} \, d \mu
\end{equation}
and
\begin{equation} \label{eq:conv_trunc_smooth_approx}
\lim_{N \to + \infty} \lim_{k \to + \infty} \int_{\Omega} T_N(u)_k^{\lambda} \, d\mu = \int_{\Omega} u^\lambda \, d \mu.
\end{equation}
Finally, for any $u \in BV(\Omega)$ we can find a sequence $(u_j)_{j \in \N} \subset C^\infty(\Omega) \cap BV(\Omega) \cap L^\infty(\Omega)$ such that
\begin{equation} \label{eq:density_conv_lambda}
\lim_{j \to + \infty} \T_{\mu, \lambda}(u_j) = \T_{\mu, \lambda}(u).
\end{equation}
\end{lemma}

\begin{proof}
We recall that $\T_{\mu, \lambda}$ is well-defined on $BV(\Omega) \cap L^\infty(\Omega)$.
Thanks to Lemma \ref{lem:2a2b-bis}, $|\mu| \in BV(\Omega)^*$ implies $u^\lambda, \mul \in L^{1}(\Omega; |\mu|)$. Hence, for any $u \in BV(\Omega)$ we get \eqref{eq:conv_trunc_lambda} by Lebesgue's Dominated Convergence Theorem with respect to the measure $|\mu|$, thanks to \eqref{eq:TuLambda_estimate}. In addition, the functional
$$\T_{\mu, \lambda}(u) = \int_{\Omega} u^\lambda \, d \mu$$ is well-defined for any $u \in BV(\Omega)$.
Let now $N > 0$ and $T_{N}(u)^{\lambda}_{k}$ be the smooth approximation of $T_{N}(u)$ given by Theorem \ref{thm:smooth_lambda_approx}. Then, we have
\begin{equation*}
T_{N}(u)^{\lambda}_{k}(x) \to T_{N}(u)^{\lambda}(x) \ \text{ as } k \to + \infty \ \text{ for } |\mu|\text{-a.e. } x \in \Omega,
\end{equation*}
since $\mu \in \MH(\Omega)$, and $|T_{N}(u)_{k}^{\lambda}(x)| \le 2 N$ for every $x \in \Omega$ and $k\ge 1$. Therefore we obtain
\begin{equation*}
\lim_{k \to \infty} \int_{\Omega} T_{N}(u)_{k}^{\lambda} \, d \mu = \int_{\Omega} T_{N}(u)^{\lambda} \, d \mu
\end{equation*}
by Lebesgue's Dominated Convergence Theorem with respect to $|\mu|$, since the constant $N$ is a summable majorant, given that $|\mu|(\Omega) < \infty$. Hence, combining this result with \eqref{eq:conv_trunc_lambda}, we get
\begin{equation*}
\lim_{N \to + \infty} \lim_{k \to + \infty} \int_{\Omega} T_N(u)_k^{\lambda} \, d\mu = \lim_{N \to + \infty} \int_{\Omega} T_N(u)^\lambda \, d \mu = \int_{\Omega} u^\lambda \, d \mu.
\end{equation*}
This proves \eqref{eq:density_conv_lambda}. Consequently, there exists a sequence $(k_j)_{j \in \N} \subset \N$ such that $u_j := T_{j}(u)_{k_j}^\lambda$ belongs to $C^\infty(\Omega) \cap BV(\Omega) \cap L^\infty(\Omega)$ and satisfies \eqref{eq:conv_trunc_smooth_approx}. 
\end{proof}

\begin{remark}
It is easy to see that, if $\mu \in BV(\Omega)^*$ and $\mu \ge 0$, then $\mu$ is an admissible measure. Hence, if we choose $\lambda \equiv \frac{1}{2}$, Lemma \ref{lem:lambda_int_approx} provides an extension of \cite[Theorem 7.4 (iv) and Remark 8.3]{Phuc_Torres}, with no additional assumptions on $\Omega$.
\end{remark}

We exploit the previous result to find a sharp estimate of $|\T_{\mu, \lambda}(u)|$ in terms of the total variation of $u$ and its trace on the boundary of the domain.

\begin{proposition}\label{prop:trunc-smooth-coarea}
Let $\Omega$ be an open bounded set with Lipschitz boundary. Let $L > 0$ and $\mu \in \PB_L(\Omega)$. Let $u \in BV(\Omega)$ and $\lambda:\Omega\to [0,1]$ be a Borel function. If $u^\lambda \in L^{1}(\Omega; |\mu|)$, then
\begin{equation} \label{eq:mu_u_Du_condition}
\left | \int_{\Omega} u^{\lambda} \, d \mu \right | \le L \left(|Du|(\Omega) + \int_{\de \Omega} |{\rm Tr}_{\partial \Omega}(u)|\, d\Haus{n-1}\right).
\end{equation}
In particular, \eqref{eq:mu_u_Du_condition} holds true for all $u \in BV(\Omega)$ as long as $\mu$ is admissible.
\end{proposition}

\begin{proof}
We start by noticing that, thanks to Remark \ref{necessary_cond_ball}, we know that $\mu \in \MH(\Omega)$. In addition, by Remark \ref{rem:summability_lambda_equiv}, we deduce that $\mul  \in L^{1}(\Omega; |\mu|)$, given that $u^\lambda \in L^{1}(\Omega; |\mu|)$. In particular, if $\mu$ is admissible, Lemma \ref{lem:2a2b-bis} ensures that $u^\lambda, \mul  \in L^{1}(\Omega; |\mu|)$ for all $u \in BV(\Omega)$. Take now $N > 0$ and let $T_{N}(u)^{\lambda}_{k}$ be the smooth approximation of $T_{N}(u)$ given by Theorem \ref{thm:smooth_lambda_approx}. 
Arguing as in the proof of Lemma \ref{lem:lambda_int_approx}, we see that
\begin{equation} \label{eq:conv_approx_trunc_lambda}
\int_{\Omega} T_{N}(u)^{\lambda} \, d \mu = \lim_{k \to \infty} \int_{\Omega} T_{N}(u)_{k}^{\lambda} \, d \mu.
\end{equation}
%by Lebesgue's Dominated Convergence Theorem with respect to $|\mu|$, since the constant $N$ is a summable majorant, given that $|\mu|(\Omega) < \infty$. 
We notice now that $T_{N}(u)_{k}^{\lambda} \in C^{\infty}(\Omega)$, so that the superlevel and sublevel sets 
\begin{equation*}
\{ T_{N}(u)_{k}^{\lambda} > t \} = \{x \in \Omega : T_{N}(u)_{k}^{\lambda}(x) > t \}  \text{ and } \{ T_{N}(u)_{k}^{\lambda} < t \} = \{x \in \Omega : T_{N}(u)_{k}^{\lambda}(x) < t \}
\end{equation*}
are open sets with smooth boundary for $\Leb{1}$-a.e. $t \in \R$, by Morse-Sard lemma \cite[Lemma 13.15]{maggi2012sets}. In particular, we deduce that, for $\Leb{1}$-a.e. $t \in \R$, 
\begin{equation*}
\{ T_{N}(u)_{k}^{\lambda} > t \}^{1} = \{ T_{N}(u)_{k}^{\lambda} > t \} \text{ and }
\{ T_{N}(u)_{k}^{\lambda} < t \}^{1} = \{ T_{N}(u)_{k}^{\lambda} < t \}.
\end{equation*}
We exploit now Theorem \ref{thm:smooth_lambda_approx} to notice that 
\begin{equation} \label{eq:unif_est_trunc_k_lambda}
|T_{N}(u)_{k}^{\lambda}| \le 2N \text{ for all } k \ge 1,
\end{equation}
\begin{equation} \label{eq:strict_conv_trunc_N_lambda}
T_{N}(u)_{k}^{\lambda} \to T_{N}(u) \ \text{ in } BV(\Omega)\text{-strict as } k \to +\infty,
\end{equation}
and 
\begin{equation} \label{eq:trace_truncation_k}
{\rm Tr}_{\de \Omega} \left (T_{N}(u)_{k}^{\lambda} \right )(x) =  {\rm Tr}_{\de \Omega} (T_{N}(u))(x) \ \text{ for } \Haus{n-1}\text{-a.e. } x \in \partial \Omega. 
\end{equation}
Therefore, by using \eqref{eq:conv_approx_trunc_lambda}, the layer-cake representation together with \eqref{eq:unif_est_trunc_k_lambda}, the fact that $\mu$ satisfies \eqref{eq:subcritical_cond}, the coarea formula, Lemma \ref{lem:slicetrace}, the convergence \eqref{eq:strict_conv_trunc_N_lambda} and the identity \eqref{eq:trace_truncation_k}, we get
\begin{align*}
\left | \int_{\Omega} T_{N}(u)^{\lambda} \, d \mu \right | & = \lim_{k \to \infty} \left | \int_{0}^{\infty} \int_{\{T_{N}(u)_{k}^{\lambda} > t \}} \, d \mu \, dt - \int_{- \infty}^{0} \int_{\{ T_{N}(u)_{k}^{\lambda} < t \}} \, d \mu \, dt \right | \\
& \le \limsup_{k\to\infty} \int_{0}^{2N} \left | \mu ( \{ T_{N}(u)_{k}^{\lambda} > t \}) \right | \, dt + \int_{-2N}^{0} \left | \mu ( \{ T_{N}(u)_{k}^{\lambda} < t \}) \right | \, dt \\
& \le L \limsup_{k\to\infty} \int_{0}^{2N} \Per( \{ T_{N}(u)_{k}^{\lambda} > t \}) \, dt + \int_{-2N}^{0} \Per ( \{ T_{N}(u)_{k}^{\lambda} < t \}) \, dt \\
& \le L \limsup_{k\to\infty} \int_{-2N}^{2N} |D \chi_{\{ T_{N}(u)_{k}^{\lambda} > t \}}|(\Omega)\, dt + \int_{0}^{2N}\int_{\de \Omega} {\rm Tr}_{\de\Omega}(\chi_{\{ |T_{N}(u)_{k}^{\lambda}| > t \}})\, d\Haus{n-1} \, d t \\
& \le L \limsup_{k\to\infty} \int_{-\infty}^{\infty} |D \chi_{\{ T_{N}(u)_{k}^{\lambda} > t \}}|(\Omega)\, dt + \int_{0}^{+\infty}\int_{\de \Omega} {\rm Tr}_{\de\Omega}(\chi_{\{ |T_{N}(u)_{k}^{\lambda}| > t \}})\, d\Haus{n-1} \, d t \\
& \le L \lim_{k\to\infty} |DT_{N}(u)_{k}^{\lambda}|(\Omega) + \int_{\de \Omega} |{\rm Tr}_{\de \Omega}(T_{N}(u)_{k}^{\lambda})|\, d\Haus{n-1}\\
& = L \left(|DT_{N}(u)|(\Omega) +  \int_{\de \Omega} |{\rm Tr}_{\de \Omega}(T_N(u))|\, d\Haus{n-1}\right)\\
& \le L \left(|Du|(\Omega) + \int_{\de \Omega} |{\rm Tr}_{\de \Omega}(T_N(u))|\, d\Haus{n-1}\right),
\end{align*}
since $|DT_{N}(u)| \le |Du|$. In addition, we notice that $T_N(u) \to u$ in $BV(\Omega)$-strict as $N \to + \infty$, so that we obtain 
$${\rm Tr}_{\de \Omega}(T_N(u)) \to {\rm Tr}_{\de \Omega}(u) \ \text{ in } L^1(\partial \Omega; \Haus{n-1}) \ \text{ as } N \to + \infty,$$
see for instance to \cite[Theorem 3.88]{AFP}. 
Thus, we obtain \eqref{eq:mu_u_Du_condition} by passing to the limit as $N \to + \infty$ also on the left hand side, again thanks to the Lebesgue's Dominated Convergence Theorem, employing \eqref{eq:TuLambda_estimate}, \eqref{eq:precise_repr_truncated_lambda2}, and the fact that $\mul \in L^{1}(\Omega; |\mu|)$.
\end{proof}

\begin{remark}
We stress that we could not directy apply the Cavalieri representation formula to $\displaystyle \int_{\Omega} T_{N}(u)^{\lambda} \, d \mu$, since, a priori, we do not know whether 
\begin{equation*}
\{ T_{N}(u)^{\lambda} > t \} = \{ T_{N}(u)^{\lambda} > t \}^{1} \ \text{ or } \ \{ T_{N}(u)^{\lambda} > t \} = \{ T_{N}(u)^{\lambda} > t \}^{1} \cup \partial^{*} \{ T_{N}(u)^{\lambda} > t \}.
\end{equation*}
Therefore, we would not be able to apply the property \eqref{eq:subcritical_cond} and proceed with the proof. Instead, we see here the usefulness of Theorem \ref{thm:smooth_lambda_approx} in ensuring that, for $\Leb{1}$-a.e. $t \in \R$, 
\begin{align*}
\{ (T_{N}(u))_{k}^{\lambda} > t \}^{1} = \{ (T_{N}(u))_{k}^{\lambda} > t \}, \ & \ \partial^{*} \{ (T_{N}(u))_{k}^{\lambda} > t \} = \partial \{ (T_{N}(u))_{k}^{\lambda} > t \}, \\
\{ (T_{N}(u))_{k}^{\lambda} < t \}^{1} = \{ (T_{N}(u))_{k}^{\lambda} < t \}, \ &  \ \partial^{*} \{ (T_{N}(u))_{k}^{\lambda} < t \} = \partial \{ (T_{N}(u))_{k}^{\lambda} < t \}.
\end{align*}
\end{remark}

\begin{lemma} \label{lem:existence_optimal_T}
Let $\Omega$ be a bounded open set with Lipschitz boundary. Let $L > 0$ and $\mu \in \PB_L(\Omega)$ be an admissible measure. Then there exists $F \in \DM^{\infty}(\Omega)$ such that $\div F = \mu$ on $\Omega$ and $\|F\|_{L^{\infty}(\Omega; \R^n)} \le L$.
\end{lemma}

\begin{proof}
Thanks to Lemma \ref{lem:2a2b-bis}, for all $u \in BV(\Omega)$ and Borel functions $\lambda : \Omega \to [0,1]$ we have $u^\lambda \in L^1(\Omega; |\mu|)$. Hence, Proposition \ref{prop:trunc-smooth-coarea} implies that $\mu \in W^{1,1}_0(\Omega)^*$; that is, it defines a linear continuous functional on $W^{1,1}_0(\Omega)$, since
\begin{equation*}
\left | \int_\Omega \tilde{u} \, d\mu \right | \le L \| \nabla u \|_{L^1(\Omega; \R^n)},
\end{equation*}
by \eqref{eq:mu_u_Du_condition}, given that ${\rm Tr}_{\de \Omega}(u) = 0$ and $u^\lambda(x) = u^*(x) = \tilde{u}(x)$ for $\Haus{n-1}$-a.e. $x \in \Omega$ and all Borel functions $\lambda : \Omega \to [0, 1]$. Hence, \cite[Lemma 7.3]{Phuc_Torres} yields
\begin{align*}
\|\mu\|_{W^{1,1}_0(\Omega)^*} & = \sup \left \{ \int_\Omega \tilde{u} \, d\mu : u \in W^{1,1}_0(\Omega), \| \nabla u \|_{L^1(\Omega; \R^n)} \le 1 \right \} \\
& = \min \left \{ \|G\|_{L^{\infty}(\Omega; \R^n)} : G \in L^\infty(\Omega; \R^n), \div G = \mu \text{ on } \Omega \right \}.
\end{align*}
Thus, this ends the proof.
\end{proof}

\begin{remark} \label{rmk:notstronger}
Under the same assumptions of Lemma \ref{lem:existence_optimal_T}, we notice that the measure $\mu$ actually satisfies
\begin{equation*}
|\mu( (E^{1} \cup \partial^* E) \cap \Omega)| \le L\, \Per(E)\,, 
\end{equation*}
for all sets $E \subset \Omega$ of finite perimeter in $\R^n$, coherently with properties (2) and (3) of Lemma \ref{necessary_cond}. 
Indeed, $\Omega$ is an open bounded set with Lipschitz boundary, and hence it is weakly regular. Therefore, we can choose a vector field $F$ as in Lemma \ref{lem:existence_optimal_T} and, using the Gauss-Green formula \eqref{eq:GG_2}, we immediately obtain
\begin{align*}
|\mu((E^{1} \cup \redb E) \cap \Omega)| & = |\div F((E^{1} \cup \redb E) \cap \Omega)| = \left | \int_{\redb E} {\rm Tr}^e(F, \partial^* E) \, d \Haus{n - 1} \right | \\
& \le  \|F\|_{L^{\infty}(\Omega; \R^n)} \Per(E) \le L \, \Per(E)
\end{align*}
for all sets $E \subset \Omega$ of finite perimeter in $\R^n$.

In addition, we notice that, if $\mu \in \MH(\Omega)$ is an admissible measure, then Theorem \ref{thm:GG_boundary_domain_u} (taking into account Remark \ref{rem:div_admissible_Leibniz}), \eqref{eq:pairing_estimate_lambda} and \eqref{eq:GG_boundary_domain} imply the estimate
\begin{equation*}
\left | \int_\Omega u^\lambda \, d \mu \right | \le \Lambda \left ( |Du|(\Omega) + \int_{\partial \Omega} |\Tr_{\partial \Omega}(u)| \, d \Haus{n-1} \right ),
\end{equation*}
where
$$\Lambda = \inf \{ \|F\|_{L^\infty(\Omega; \R^n)} : F \in \DM^\infty(\Omega) \text{ with } \div F = \mu \}.$$
However, assuming $\mu \in \PB_L(\Omega)$ for some $L > 0$, we cannot prove that $\Lambda \le L$ without exploiting Proposition \ref{prop:trunc-smooth-coarea} itself, since in the proof of Lemma \ref{lem:existence_optimal_T} it is essential to know that the perimeter bound for sets holds at least for functions in $W^{1,1}_0(\Omega)$.
\end{remark}

As a final result, we show the stability of the admissibility and the perimeter bound conditions under a suitable type of smooth approximation procedure. More precisely, we prove that any admissible measure $\mu \in \PB_{L}(\Omega)$ can be approximated in the weak--$\ast$ sense by a sequence of absolutely continuous measures $\mu_j \Leb{n} \in \PB_{L_j}(\Omega)$ with smooth density functions, where $L_j \to L$ as $j \to + \infty$. The key idea is to use the duality between admissible measures in $\PB(\Omega)$ and essentially bounded divergence-measure fields. While it is not clear if a direct mollification of $\mu$ would in general still belong to $\PB_L(\Omega)$ (even assuming $\mu$ of compact support in $\Omega$), we shall instead rely on the fact that the Anzellotti-Giaquinta--type regularization of a vector field almost preserves its $L^{\infty}$ norm. Similar ideas can be found for instance in \cite{dal1999renormalized}.

\begin{proposition}\label{prop:muapprox}
Let $\Omega$ be an open bounded set with Lipschitz boundary.
Let $L > 0$ and $\mu \in \PB_L(\Omega)$ be an admissible measure. Then there exists a sequence of functions $(\mu_{j})_{j} \subset C^{\infty}(\Omega)$ such that $\mu_j \Leb{n}$ is an admissible measure and it belongs to $\PB_{L(1 + 1/j)}(\Omega)$ for all $j \in \N$, and $\mu_j \, \Leb{n} \weakto \mu$ as $j \to + \infty$.
Moreover, if $\mu = h\, \Leb{n} + \mu^{s}$ with $h\in L^{q}(\Omega)$, for some $q \in [1, + \infty]$, and $\mu^{s}$ has compact support in $\Omega$, then for any open set $\Omega'\Subset \Omega$ containing the support of $\mu^{s}$ we have the following properties:
\begin{itemize}
\item there exists $(h_{j})_j \subset C^{\infty}(\Omega)$ such that $h_{j}\to h$ in $L^{p}(\Omega)$, for all $p \in [1, q)$ and $p = q$ if $q < + \infty$, and $\mu_{j} - h_{j} \to 0$ in $L^n(\Omega\setminus \Omega')$;
\item $\mu_{j} = \rho_{j}\ast \mu^s + h_j$ on $\Omega'$ for $j$ large enough and $\rho_j(x) = \delta_{j}^{-n} \rho(x/\delta_j)$, for some sequence $\delta_j \downarrow 0$ and some fixed mollifier $\rho \in C^{\infty}_c(B_1)$.
\end{itemize}
\end{proposition}

\begin{proof}
The construction of the sequence $\mu_{j}$ starting from the measure $\mu$ can be done as follows. We recall that, thanks to Proposition \ref{prop:adm_PB_div}(4) and Lemma \ref{lem:existence_optimal_T}, if $\mu \in \PB_{L}(\Omega)$ and it is admissible, then there exist vector fields $F, G \in \DM^{\infty}(\Omega)$ such that 
\begin{itemize}
\item $\|F\|_{L^{\infty}(\Omega; \R^n)} \le L$ and $\div F = \mu$ on $\Omega$, 
\item $\div G = |\mu|$ on $\Omega$.
\end{itemize} 
Analogously as in the proof of Theorem \ref{thm:Anzellotti_Giaquinta}, we set $\Omega_{0, m} = \emptyset$ for $m \in \N$,
\[ \Omega_{k, m} = \left \{ x \in \Omega : \ \mathrm{dist}(x, \partial \Omega) > \frac{1}{m + k} \right \}  \quad \text{ for } k, m \ge 1 \]
and then $\Sigma_{k, m} := \Omega_{k+1, m} \setminus \overline{\Omega_{k-1, m}}$. For notational simplicity we set $\Omega_k = \Omega_{k, m}$ and $\Sigma_k = \Sigma_{k, m}$, since in the first part of the proof the second parameter does not play any role. We let $(\zeta_{k})_{k \ge 1}$ be a partition of unity subordinate to the open cover $(\Sigma_{k})_{k \ge 1}$ of $\Omega$. We choose a nonnegative sequence $(\delta_{j, k})_{j, k \ge 1}$ such that $\delta_{j, k} \to 0$ as $j \to + \infty$ for all $k \ge 1$ and
\begin{equation} \label{eq:supp_zeta_rho}
{\rm supp}(\zeta_{k}) + B_{\delta_{j,k}} \subset \Sigma_{k}.
\end{equation} 
Finally, we take a standard mollifier $\rho$ and we require that the family of mollifiers $\rho_{j,k} = \rho_{\delta_{j,k}}$ satisfies the following set of conditions:
\begin{equation}
\| \rho_{\delta_{j, k}} \ast \zeta_{k} - \zeta_{k} \|_{L^{\infty}(\Omega)} \le \frac{1}{j 2^{k}} \, , \label{eq:conv_cutoff_infty_1}
\end{equation}
\begin{equation} \label{eq:A3_mod}
\|\rho_{j,k}\ast (F\cdot \nabla \zeta_{k}) - F\cdot \nabla \zeta_{k}\|_{L^r(\Omega)} \le \frac{1}{j 2^k}
\end{equation}
for all $r \in [1, n]$,
\begin{equation} \label{eq:A3_mod_G}
\|\rho_{j,k}\ast (G\cdot \nabla \zeta_{k}) - G\cdot \nabla \zeta_{k}\|_{L^n(\Omega)} \le \frac{1}{j 2^k}
\end{equation}
and, if $\mu = h\, \Leb{n} + \mu^{s}$,
\begin{equation} \label{eq:A2_mod}
\|\rho_{j,k} \ast(\zeta_{k} h) - \zeta_{k} h\|_{L^p(\Omega)} \le \frac{1}{j 2^k},
\end{equation}
for $p \in [1, q)$ and $p = q$ if $q < + \infty$.
Then we set
\[
F_{j} = \sum_{k=1}^{+\infty} \rho_{j,k}\ast (\zeta_{k} F)
\]
and, finally,
\begin{equation}\label{eq:muj}
\mu_{j} = \div F_{j}.
\end{equation}
With reference to the proof of Theorem \ref{thm:Anzellotti_Giaquinta}, we fix $\varphi\in C^{\infty}_{c}(\Omega)$ and obtain
\begin{align*}
\int_\Omega \varphi \, \mu_{j}\, dx &= -\int_{\Omega} F_{j} \cdot \nabla \varphi\, dx\\
&= -\sum_{k=1}^{+\infty} \int_\Omega (\zeta_{k}F)\ast \rho_{j,k} \cdot \nabla \varphi\, dx\\
&= \sum_{k=1}^{+\infty} \int_\Omega \zeta_{k}(\rho_{j,k} \ast \varphi)\, d\mu + \sum_{k=1}^{+\infty} \int_\Omega \varphi\Big(\rho_{j,k}\ast (F\cdot \nabla \zeta_{k}) - F\cdot \nabla \zeta_{k}\Big) \, dx\\
&\longrightarrow \int_\Omega \varphi\, d\mu\,,\qquad \text{as $j\to + \infty$,}
\end{align*}
where the final passage to the limit is justified because
\begin{equation*}
\left | \int_{\Omega} \varphi\Big(\rho_{j,k}\ast (F\cdot \nabla \zeta_{k}) - F\cdot \nabla \zeta_{k}\Big) \, dx \right | \le \|\varphi\|_{L^\infty(\Omega)} \|\rho_{j,k}\ast (F\cdot \nabla \zeta_{k}) - F\cdot \nabla \zeta_{k}\|_{L^1(\Omega)} \le \frac{\|\varphi\|_{L^\infty(\Omega)}}{j 2^k}
\end{equation*}
thanks to \eqref{eq:A3_mod}. 
As a byproduct of the previous calculations, we also get
\begin{equation} \label{eq:mu_j_repr}
\mu_{j} = \sum_{k = 1}^{+\infty} \rho_{j,k}\ast (\zeta_{k} \mu) + \sum_{k=1}^{+\infty} \left(\rho_{j,k}\ast (F\cdot \nabla \zeta_{k}) - F\cdot \nabla \zeta_{k}\right)\,.
\end{equation}
In order to prove that $\mu_j \Leb{n}$ is admissible, we consider separately the two sums on the right hand side. Given that $\div G = |\mu|$, it is easy to see that
\begin{align*}
|\mu_{j}| \le \sum_{k = 1}^{+\infty} \rho_{j,k}\ast (\zeta_{k} \div G) + \sum_{k=1}^{+\infty} \left | \rho_{j,k}\ast ( F\cdot \nabla \zeta_{k}) - F\cdot \nabla \zeta_{k}\right|\,.
\end{align*}
Due to \eqref{eq:A3_mod} and the Sobolev embedding $BV(\Omega) \subset L^{\frac{n}{n-1}}(\Omega)$\footnote{If $n=1$, $BV(\Omega) \subset L^{\infty}(\Omega)$.}, for all $u \in BV(\Omega)$, we have
\begin{align*}
\left | \int_\Omega u \, \sum_{k=1}^{+\infty} \left | \rho_{j,k}\ast ( F\cdot \nabla \zeta_{k}) - F\cdot \nabla \zeta_{k}\right| \, dx \right | & \le  \sum_{k = 1}^{+\infty} \|u\|_{L^{\frac{n}{n-1}}(\Omega)} \|\rho_{j,k}\ast (F\cdot \nabla \zeta_{k}) - F\cdot \nabla \zeta_{k}\|_{L^n(\Omega)} \\
& \le C \|u\|_{BV(\Omega)}.
\end{align*}
As for the first sum, given that $\div G$ is an admissible measure, in the light of Remark \ref{rem:div_admissible_Leibniz} we can exploit the integration by parts formula \eqref{eq:GG_boundary_domain_u} to obtain
\begin{align*}
\int_\Omega u \, \sum_{k = 1}^{+\infty} \rho_{j,k}\ast (\zeta_{k} \div G) \, dx & = \sum_{k = 1}^{+\infty} \int_\Omega  \zeta_{k} (\rho_{j,k} \ast u) \, d \div G \\
& = -  \sum_{k = 1}^{+\infty} \int_\Omega G \cdot \nabla \left (  \zeta_{k} (\rho_{j,k} \ast u) \right ) \, dx,
\end{align*}
since ${\rm supp}( \zeta_{k}) \Subset \Omega$. Now, we exploit the Leibniz rule to get
\begin{align*}
\sum_{k = 1}^{+\infty} \int_\Omega G \cdot \nabla \left (  \zeta_{k} (\rho_{j,k} \ast u) \right ) \, dx & = \int_\Omega G \cdot  \left (\sum_{k = 1}^{+\infty}  \zeta_{k} (\rho_{j,k} \ast D u) \right ) \, dx + \sum_{k = 1}^{+\infty}  \int_\Omega  u \, \nabla \zeta_{k} \cdot (\rho_{j,k} \ast G - G) \, dx \\
& + \sum_{k = 1}^{+\infty}  \int_\Omega u \, \nabla \zeta_{k}  \cdot G \, dx.
\end{align*}
We deal with the three terms separately. As for the first term, we exploit the fact that $\sum_{k = 1}^{+\infty} \chi_{\Sigma_k} \le 2$ and \eqref{eq:conv_cutoff_infty_1} in order to get 
\begin{align*}
\left | \int_\Omega G \cdot  \left (\sum_{k = 1}^{+\infty}  \zeta_{k} (\rho_{j,k} \ast D u) \right ) \, dx \right | & \le \|G\|_{L^\infty(\Omega; \R^n)} \sum_{k=1}^{+\infty} \int_{{\rm supp}(\zeta_{k})} \rho_{j,k} \ast |D u| \, dx \\
& \le 2 \|G\|_{L^\infty(\Omega; \R^n)} |Du|(\Omega).
\end{align*}
Then, by \eqref{eq:A3_mod_G} we see that
\begin{equation*}
\left | \sum_{k = 1}^{+\infty} \int_\Omega  u \, \nabla \zeta_{k} \cdot (\rho_{j,k} \ast G - G) \, dx \right | \le \sum_{k=1}^{+\infty} \|u\|_{L^{\frac{n}{n-1}}(\Omega)} \|\rho_{j,k}\ast (G\cdot \nabla \zeta_{k}) - G\cdot \nabla \zeta_{k}\|_{L^n(\Omega)} \le C \|u\|_{BV(\Omega)},
\end{equation*}
again thanks to the Sobolev embedding for $BV$ functions. Finally, we exploit again \eqref{eq:GG_boundary_domain_u} to obtain
\begin{align*}
\sum_{k =1}^{+\infty} \int_\Omega u \, \nabla \zeta_{k}  \cdot G \, dx & = - \sum_{k =1}^{+\infty} \int_\Omega \zeta_{k} \, d \div (u G) = - \int_\Omega \, d \div(u G) = \int_{\partial \Omega} {\rm Tr}^{i}(G, \partial \Omega) {\rm Tr}_{\partial \Omega}(u) \, d \Haus{n-1},
\end{align*}
so that, thanks to \eqref{eq:GG_boundary_domain}, we get
\begin{equation*}
\left | \sum_{k =1}^{+\infty} \int_\Omega u \, \nabla \zeta_{k}  \cdot G \, dx \right | \le \|G\|_{L^\infty(\Omega; \R^n)} \|u\|_{BV(\Omega)}.
\end{equation*}
All in all, this shows that $|\mu_j| \Leb{n} \in BV(\Omega)^*$.

In addition, when $\mu = h\, \Leb{n} + \mu^{s}$ with $h\in L^{q}(\Omega)$, for some $q \in [1, + \infty]$, and $\mu^{s}$ has compact support in $\Omega'\Subset \Omega$, by choosing $m$ large enough so that $\Omega' \Subset \Omega_{1, m}$, thanks to \eqref{eq:supp_zeta_rho} the first sum in \eqref{eq:mu_j_repr} coincides with $\rho_{j,1}\ast \mu^s + h_j$ if restricted to $\Omega'$, where $h_j$ is the Anzellotti-Giaquinta's regularization of $h$, and it coincides with $h_{j}$ when restricted to $\Omega\setminus \Omega'$. At the same time, the second sum in \eqref{eq:mu_j_repr} tends to $0$ as $j\to\infty$ in $L^{n}(\Omega)$ by \eqref{eq:A3_mod}. Finally, $h_j \to h$ in $L^p(\Omega)$ with $p$ as above, by \eqref{eq:A2_mod}.
This shows the last part of the statement.

Finally, we prove that $\mu_{j} \in \PB_{L(1 + 1/j)}(\Omega)$. Given any set $E \subset \Omega$ of finite perimeter in $\R^n$, by Theorem \ref{thm:GG_app} we have
\begin{align*}
\left|\int_{E \cap \Omega} \mu_{j} \, dx\right| &= \left| \int_{\redb E} {\rm Tr}^i(F_j, \redb E) \, d\Haus{n-1} \right| \le \|F_{j}\|_{L^{\infty}(\Omega; \R^n)}\, \Per(E).
\end{align*}
Then, we notice that, for any $x \in \Omega$, there exists a unique $k_0 \ge 1$ such that $x \in \Sigma_{k_0} \cap \Sigma_{k_0+1}$, so that, arguing as in the proof of point (5) of Theorem \ref{thm:Anzellotti_Giaquinta}, thanks to the assumptions on $\zeta_{k}$ and $\rho_{j,k}$, in particular \eqref{eq:supp_zeta_rho} and \eqref{eq:conv_cutoff_infty_1}, we obtain
\[
|F_{j}(x)| \le \|F\|_{L^{\infty}(\Omega; \R^n)} \left(1 + \sum_{k=k_0}^{k_{0}+1} |(\zeta_{k}\ast \rho_{j,k})(x) - \zeta_{k}(x)|\right) \le L(1+1/j)\,,
\]
\end{proof}

\end{document}